\documentclass[reqno]{amsart}

\usepackage{amsmath,amssymb,amsthm}
\usepackage{graphicx,tikz}
\newtheorem{theorem}{Theorem}

\newtheorem{proposition}{Proposition}

\newtheorem*{definition}{Definition}

\theoremstyle{remark}

\DeclareMathOperator{\vol}{vol}

\DeclareMathOperator{\sgn}{sign}

\DeclareMathOperator{\am}{\amalg}

\begin{document}

\title[]{Quantum entanglement and the \\ Growth of Laplacian Eigenfunctions}
\author[]{Stefan Steinerberger}
\address{Department of Mathematics, University of Washington, Seattle, WA 98195, USA}
\email{steinerb@uw.edu}

\begin{abstract} We study the growth of Laplacian eigenfunctions $ -\Delta \phi_k = \lambda_k \phi_k$ on compact manifolds $(M,g)$.  H\"ormander proved sharp polynomial bounds on $\| \phi_k\|_{L^{\infty}}$ which are attained on the sphere. On a `generic' manifold, the behavior seems to be different: both numerics and Berry's random wave model suggest $\| \phi_k\|_{L^{\infty}} \lesssim \sqrt{\log{\lambda_k}}$ as the  typical behavior.  We propose a mechanism, centered around an $L^1-$analogue of the spectral projector, for explaining the slow growth in the generic case: for $\phi_{n+1}(x_0)$ to be large, it is necessary that either (1) several of the first $n$ eigenfunctions were large in $x_0$ or (2) that $\phi_{n+1}$ is strongly correlated with a suitable linear combination of the first $n$ eigenfunctions on most of the manifold or (3) both.  An interesting byproduct is quantum entanglement for Laplacian eigenfunctions: the existence of two distinct points $x,y \in M$ such that the sequences $(\phi_k(x))_{k=1}^{\infty}$ and $(\phi_k(y))_{k=1}^{\infty}$ do \textit{not} behave like independent random variables. The existence of such points is not to be expected for generic manifolds but common for the classical manifolds and subtly intertwined with eigenfunction concentration.\end{abstract}

\maketitle

\section{Introduction}
Let $(M,g)$ be a compact $d-$dimensional manifold normalized to $\vol(M) = 1$. We consider the sequence of Laplacian eigenfunctions 
\begin{equation}
 -\Delta \phi_k = \lambda_k \phi_k.
 \end{equation}
Throughout the paper it is assumed that they are $L^2-$normalized, i.e. $\| \phi_k\|_{L^2} = 1$. One of the most basic questions is to understand how much these eigenfunctions can concentrate in a point:  to understand the behavior of $\| \phi_k\|_{L^{\infty}}$. A classic result of H\"ormander \cite{h} (see also 
 Avakumovic \cite{ava}, Grieser \cite{grieser}, Levitan \cite{levitan}, Sogge \cite{sogge2002}) is
\begin{equation}
 \| \phi_k\|_{L^{\infty}} \lesssim_{(M,g)} \lambda_k^{\frac{d-1}{4}},
 \end{equation}
 where $A \lesssim B$ denotes the existence of a constant $c>0$ such that $A \leq cB$.
This inequality is sharp and attained for the $d-$dimensional sphere $\mathbb{S}^d$. 
One way of seeing this is via local Weyl laws (see \cite{h}): for any $x \in M$, we have
\begin{equation}
 \sum_{k=1}^{n} \phi_k(x)^2 = n + \mathcal{O}(n^{\frac{d-1}{d}}).
 \end{equation}

The sum under consideration is the diagonal of the spectral projector 
$$ \Pi(x,y) = \sum_{k=1}^{n} \phi_k(x) \phi_k(y).$$
 We also have, for $x \neq y$, that
$$ \left| \Pi(x,y) \right| =  \left| \sum_{k=1}^{n} \phi_k(x) \phi_k(y) \right| \lesssim  n^{\frac{d-1}{d}}.$$
At this point it seems like the problem has been completely resolved: the maximal rate of growth of $\| \phi_k\|_{L^{\infty}}$
is polynomial and attained on the sphere $\mathbb{S}^d$. However, it seems as if the behavior on the sphere is
actually a very exceptional case. If we take some `generic' manifold (say, without any symmetries or an arbitrary domain subjected to a generic diffeomorphism), then numerical
experiments indicate that $\| \phi_k\|_{L^{\infty}}$ tends to grow only very slowly (see e.g. \cite{aurich}).  It is assumed that the growth is perhaps only logarithmic: a guess sometimes mentioned is
$ \| \phi_k\|_{L^{\infty}} \lesssim \sqrt{\log \lambda_k}.$
In contrast, on manifolds on which the eigenvalue problem is explicitly solvable, we frequently encounter
eigenfunction growth. On the $d-$dimensional torus $\mathbb{T}^d$ and $d \geq 5$, classical results from number theory imply
$$ \| \phi_k\|_{L^{\infty}} \lesssim \lambda_k^{\frac{d-2}{4}}$$
and this bound is best possible (see Bourgain \cite{bour}). Bourgain \cite{bour2} also proved that the flat metric on $\mathbb{T}^2$ can be
perturbed to yield a sequence of eigenfunctions $-\Delta \phi_k = \lambda_k \phi_k$ such that $\lambda_k = k^2 + \mathcal{O}(k)$ and $\| \phi_k\|_{L^{\infty}} \gtrsim \lambda_k^{1/8}$. A similar such exceptional sequence on an arithmetic hyperbolic manifold was constructed by Rudnick \& Sarnak \cite{rudnick}.
In the converse direction, there have been several results regarding conditions under which H\"ormander's estimate can be improved, we refer to work of B\'erard \cite{ber},  Galkowski \& Toth \cite{gal}, Hassell \& Tacy \cite{ha}, Hezari \& Rivi\`ere \cite{hezari}, Sogge \cite{sogge0, sogge00, sogge1}, Sogge \& Zelditch \cite{sogge, soggez, soggez2} and Xi \& Zhang \cite{xi}. A common theme of these results is that, under some geometric conditions excluding at least the sphere, they obtain a logarithmic improvement.\\

\textbf{The Random Wave Model.} A prediction of Berry \cite{berry} is that eigenfunctions should locally (say, on the scale of a few wavelengths) behave like a superposition of random waves, their behavior should not be too different from
$$ f(x) = \sqrt{\frac{2}{\mbox{vol}(M)}} \frac{1}{\sqrt{N}} \sum_{n=1}^{N} a_n \cos \left( \left\langle k_n, x \right\rangle + \varepsilon_n \right),$$
where $a_n$ are independent Gaussians, the $\varepsilon_n$ are uniformly distributed in $[0,2\pi]$ and the vectors $k_n$ are chosen uniformly from the sphere with radius $\| k_n\| = \sqrt{\lambda}$. This proposal has received a lot of attention and leads to predictions of local behavior
that match numerical results -- locally, in the generic setting, Laplacian eigenfunctions seem to behave like random waves. This can be used to predict the `typical' $L^{\infty}-$norm: the global maximum should simply be the largest of a polynomial number of waves: the maximum of $m$ standard $\mathcal{N}(0,1)-$gaussians scales like $\sim \sqrt{\log{m}}$.
Thus, on manifolds where eigenfunctions are well captured by the random wave model, the asymptotic behavior is perhaps given by 
$$\| \phi_k\|_{L^{\infty}} \lesssim \sqrt{\log \lambda_k}.$$ In the other direction, Toth \& Zelditch \cite{toth} have established that a uniform bound $\| \phi_k\|_{L^{\infty}} \lesssim 1$ requires (under some assumptions) the manifold $(M,g)$ to be flat and thus one would perhaps not expect this to be generic.

\section{Statements and Results} 
\subsection{Introducing $\am$} The aim of this paper is to introduce, $\am$, an object related to the spectral projector $\Pi$ but rougher. We will argue that it has interesting properties and that these properties can be used to study the growth of eigenfunctions. Given the first $n$ eigenfunctions $\phi_1, \dots, \phi_n$, we define $\am:M \times M \rightarrow \mathbb{R}$ via
\begin{equation}
 \am^{(n)}(x,y) = \sum_{k=1}^n \sgn(\phi_k(x)) \phi_k(y),
 \end{equation}
where 
$$ \sgn(z) = \begin{cases} 1 \qquad &\mbox{if}~z > 0 \\
0 \qquad &\mbox{if}~z =0 \\
 -1 \qquad &\mbox{if}~z < 0. \end{cases}$$
 Observe that the eigenfunctions $\phi_1, \dots, \phi_n$ are only defined up to a global sign: if $\phi_k$ is an eigenfunction, then $-\phi_k$ is just as good an eigenfunction as $\phi_k$ is. $\am^{(n)}$ is \textit{invariant} under these changes of signs.
$\am^{(n)}$ depends on the number $n$ of eigenfunctions being used. We will often suppress this dependence in the notation for simplicity of exposition and  write $\am^{(n)}(x,y)$ when trying to emphasize the dependence on $n$. $\am(x,y)$ is a much rougher function than $\Pi$: in particular, it is discontinuous in the first variable. However, for fixed $x_0 \in M$, the function $\am^{(n)}(x_0, y)$ is merely a sum of eigenfunctions since
 $$ \am^{(n)}(x_0,y) = \varepsilon_1 \phi_1(y) + \varepsilon_2 \phi_2(y) + \dots + \varepsilon_n \phi_n(y) \qquad \mbox{where} ~\varepsilon_i \in \left\{-1,0,1\right\}$$
 is chosen in such a way that $\varepsilon_i \phi_i(x) \geq 0$ for all $1 \leq i \leq n$ in the point $x \in M$. 
 One way of thinking about $\am^{(n)}(x,y)$ is that its the sum of the first $n$ eigenfunctions where each of these eigenfunctions has had its sign flipped to ensure that it is positive in $x$ (while removing all eigenfunctions that vanish in $x$).
 
 \begin{center}
 \begin{figure}[h!]
 \begin{tikzpicture}
 \node at (0,0) {\includegraphics[width =0.45\textwidth]{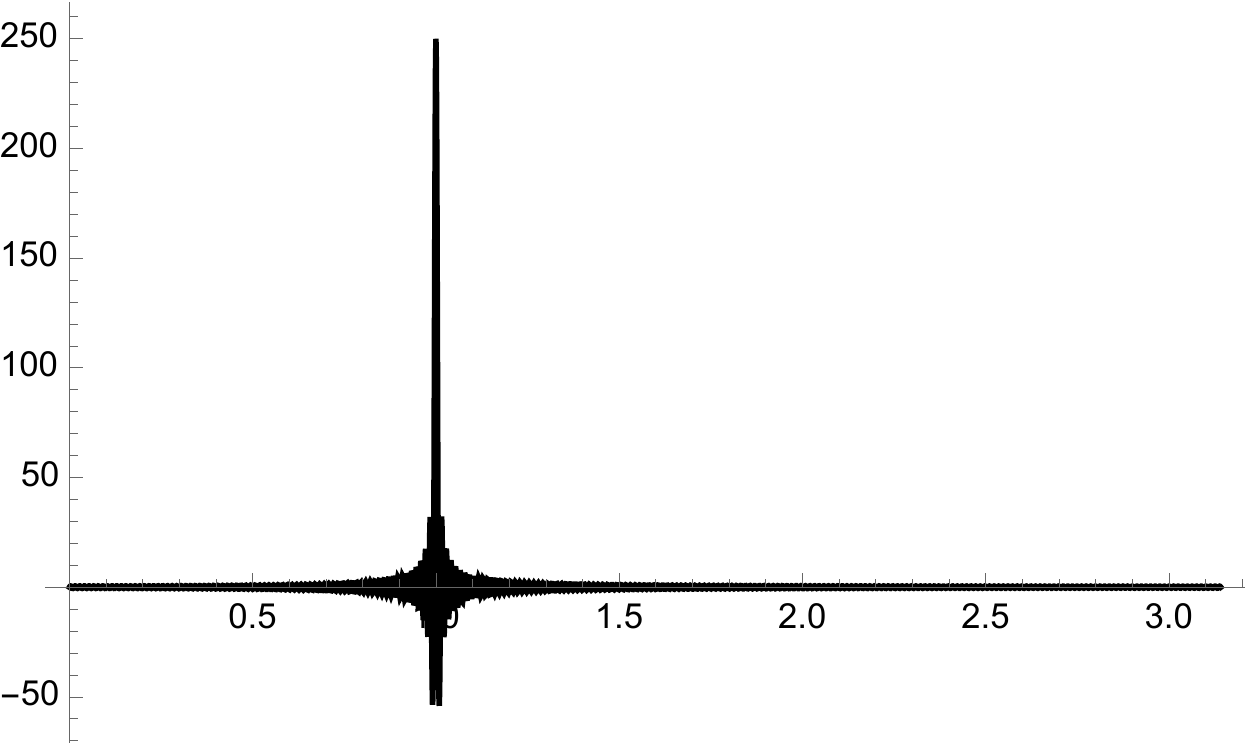}};
 \node at (1,1) {$\Pi(1, y)$};
  \node at (6,0) {\includegraphics[width =0.45\textwidth]{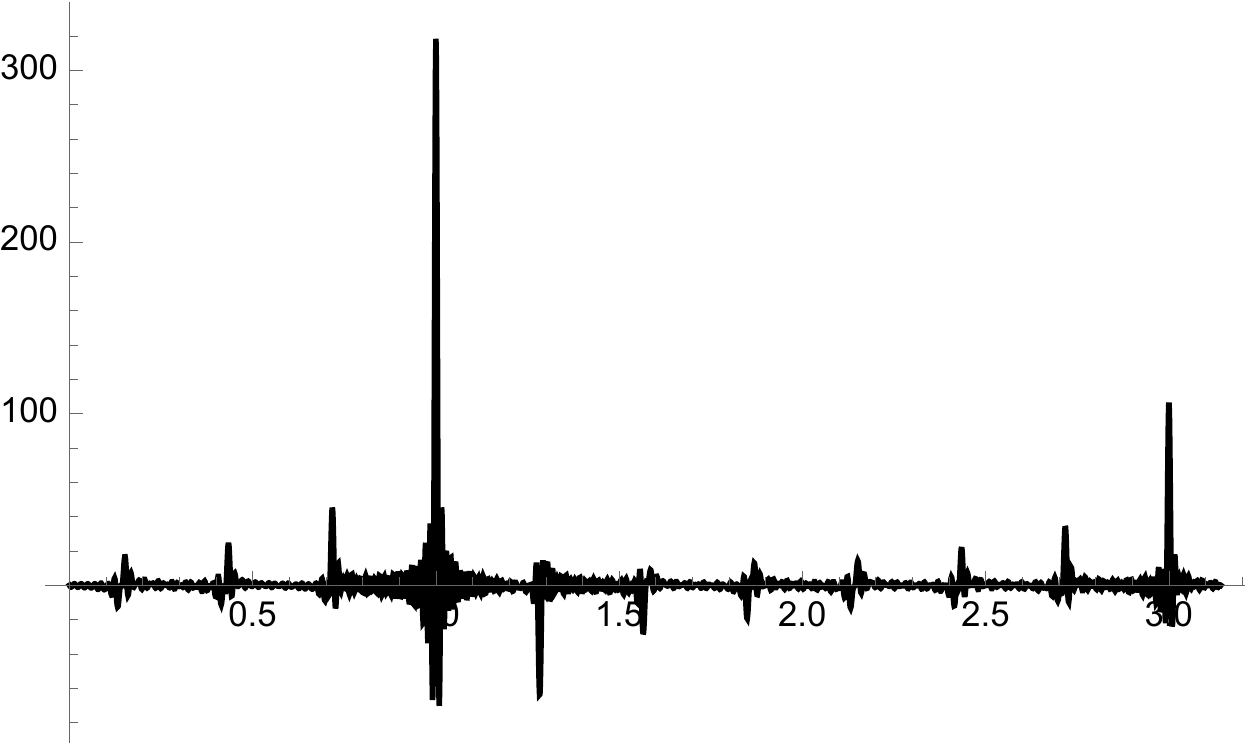}};
   \node at (7,1) {$\am(1, y)$};
 \end{tikzpicture}
 \caption{An example: the first 500 eigenfunctions of the Laplacian on $[0, \pi]$ with Dirichlet boundary conditions. The spectral projector $\Pi(1,y)$ is tightly concentrated around 1. $\am(1,y)$ has some concentration around 1 but also in other places (see \S 5.2).}
 \end{figure}
 \end{center}
 
 We start with a number of basic facts. A direct consequence of the orthogonality and $L^2-$normalization of the eigenfunction immediately implies that for most pairs of points $x,y \in M$ we expect an inequality of the type $|\am^{(n)}(x,y)| \lesssim \sqrt{n}$.
 \begin{proposition}
 We have
\begin{align}
\forall x \in M \qquad \int_M \am^{(n)}(x,y)^2 dy &\leq n \qquad \mbox{and} \qquad \int_M\int_M \am^{(n)}(x,y)^2 dy = n. \label{eq:up}
\end{align}
 \end{proposition}
We will often be particularly interested in the diagonal behavior which is
 $$ \am^{(n)}(x,x) = \sum_{k=1}^{n} |\phi_k(x)|.$$
One might expect that for $x \in M$ fixed, a typical eigenfunction $\phi_k(x)$ should be of size $\sim 1$ and that we should be
 able to expect $\am(x,x) \sim n$. This is indeed the correct upper bound. However, $\am^{}(x,x)$ can
 be much smaller (exactly in the case where eigenfunctions exhibit large concentration).  
 \begin{proposition}[Diagonal Behavior of $\am$]
Assuming $\vol(M) = 1$, 
\begin{equation} \label{eq:l1l}
\int_{M}  \am^{(n)}(x,x) dx =    \int_{M}  \sum_{k=1}^{n} |\phi_k(x)| dx \leq n.
 \end{equation}
Moreover, for each $x \in M$, we have 
\begin{equation} \label{eq:univlower}
 n^{\frac{d+1}{2d}} \lesssim_{(M,g)} \sum_{k=1}^{n} |\phi_k(x)|  \leq  n + \mathcal{O}(n^{\frac{d-1}{d}}).
 \end{equation}
\end{proposition}
The lower bound in \eqref{eq:univlower} is sharp and attained on the sphere.
We remark that the local Weyl law implies that in the absence of eigenfunction concentration, $\am^{}(x,x)$ is uniformly close to $n$ since
\begin{equation} \label{eq:easylower}
 \min_{x \in M} \am^{(n)}(x,x) \geq \frac{ n - \mathcal{O}(n^{\frac{d-1}{d}})}{\max_{1 \leq k \leq n} \| \phi_k\|_{L^{\infty}}}.
 \end{equation}
 In particular, if in the generic setting, the eigenfunctions can only grow logarithmically, then $\am^{}(x,x)$ is roughly comparable to $n$ up to possibly logarithmic multiplicative factors. 
The $L^2-$normalization also implies an average bound
\begin{equation} \label{eq:easylower2}
 \int_{x \in M} \am^{(n)}(x,x)dx \geq \sum_{k=1}^{n} \frac{1}{\| \phi_k \|_{L^{\infty}}}.
 \end{equation}

We conclude the section with a seemingly natural question regarding \eqref{eq:l1l}: one might perhaps
expect that such an upper bound should have complementary lower bound of nearly comparable quality
but we have been unable to find such a statement in the literature:
for a sequence of Laplacian eigenfunctions $(\phi_k)_{k=1}^{\infty}$ on a compact manifold $(M,g)$, normalized in $L^2$, is there an estimate along the lines of
\begin{equation} \label{lowerL1}
\sum_{k=1}^{n} \| \phi_k\|_{L^{1}} \gtrsim \frac{n}{(\log{n})^{\alpha}}
\end{equation}
for some $\alpha \geq 0$? The random wave model in combination with the prediction of logarithmic growth for $\phi_k$ in combination with \eqref{eq:easylower} would suggest that
such an inequality should be true for some $0 \leq \alpha \leq 1/2$.

\subsection{Spooky Action at a Distance.} We will now discuss what one might expect of $\am^{}(x,y)$. A common theme that will
be illustrated throughout the paper (see also the various figures) is that on manifolds $(M,g)$ where explicit eigenfunction computations are possible, $\am^{}(x,y)$
exhibits a remarkable amount of structure. The purpose of this paper is to introduce and motivate $\am(x,y)$ and its connection to eigenfunction
concentration -- the paper does not contain a systematic study of $\am(x,y)$ on specific manifolds beyond a few basic examples (see Theorem 1); such a study might be interesting and already nontrivial on, say, $\mathbb{S}^1$ or $\mathbb{T}^2$.
As can be seen in Fig. 1 and Fig. 2, $\am(x,y)$ can have significant nonlocal behavior. A first guess is that there is off-diagonal decay: recalling \eqref{eq:up}
$$\int_M\int_M \am^{(n)}(x,y)^2 dy = n$$ 
one would expect that if $x \neq y$ and $n \in \mathbb{N}$, then
$$ \left| \am^{(n)}(x,y) \right| \lesssim \sqrt{n}.$$
Since $\am(x,y)$ fixes the signs at $x$, one would perhaps expect that the signs at a point $y \neq x$ are completely decoupled and should be independent; the law of iterated logarithm would then suggest the following optimistic estimate 
$$ \left| \am^{(n)}(x,y) \right| \lesssim_{x,y} \sqrt{n \log \log n}.$$
However, this fails \textit{dramatically} for most of the classical examples. Most of them seem to have points $x \neq y$ such that $|\am^{(n)}(x,y)|$ is dramatically larger for infinitely many $n$ -- this is what we call spooky action at a distance. For the formal definition, we will only work up to powers of logarithms.

 \begin{center}
 \begin{figure}[h!]
 \begin{tikzpicture}
 \node at (0,0) {\includegraphics[width=0.35\textwidth]{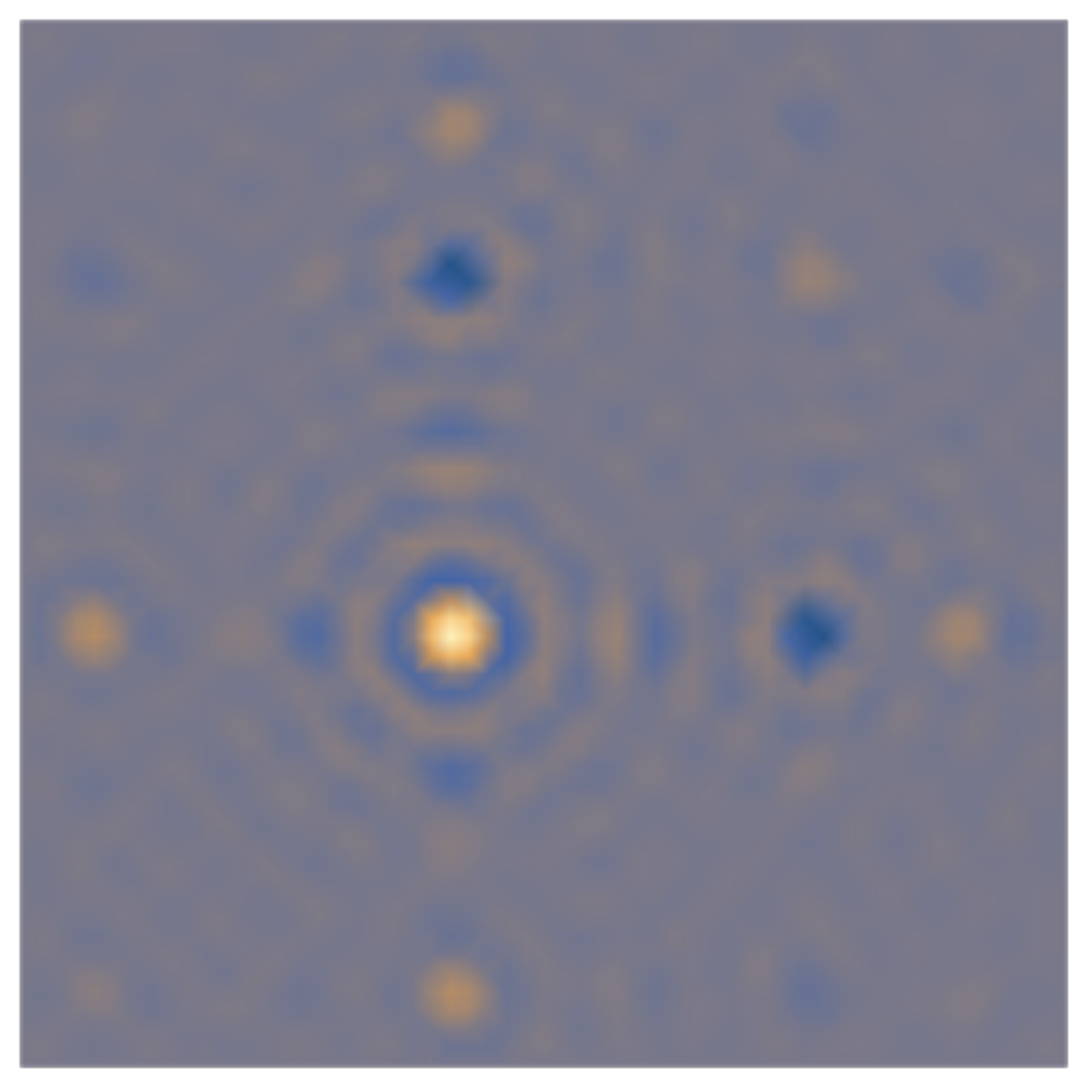}};
  \node at (6,0) {\includegraphics[width=0.35\textwidth]{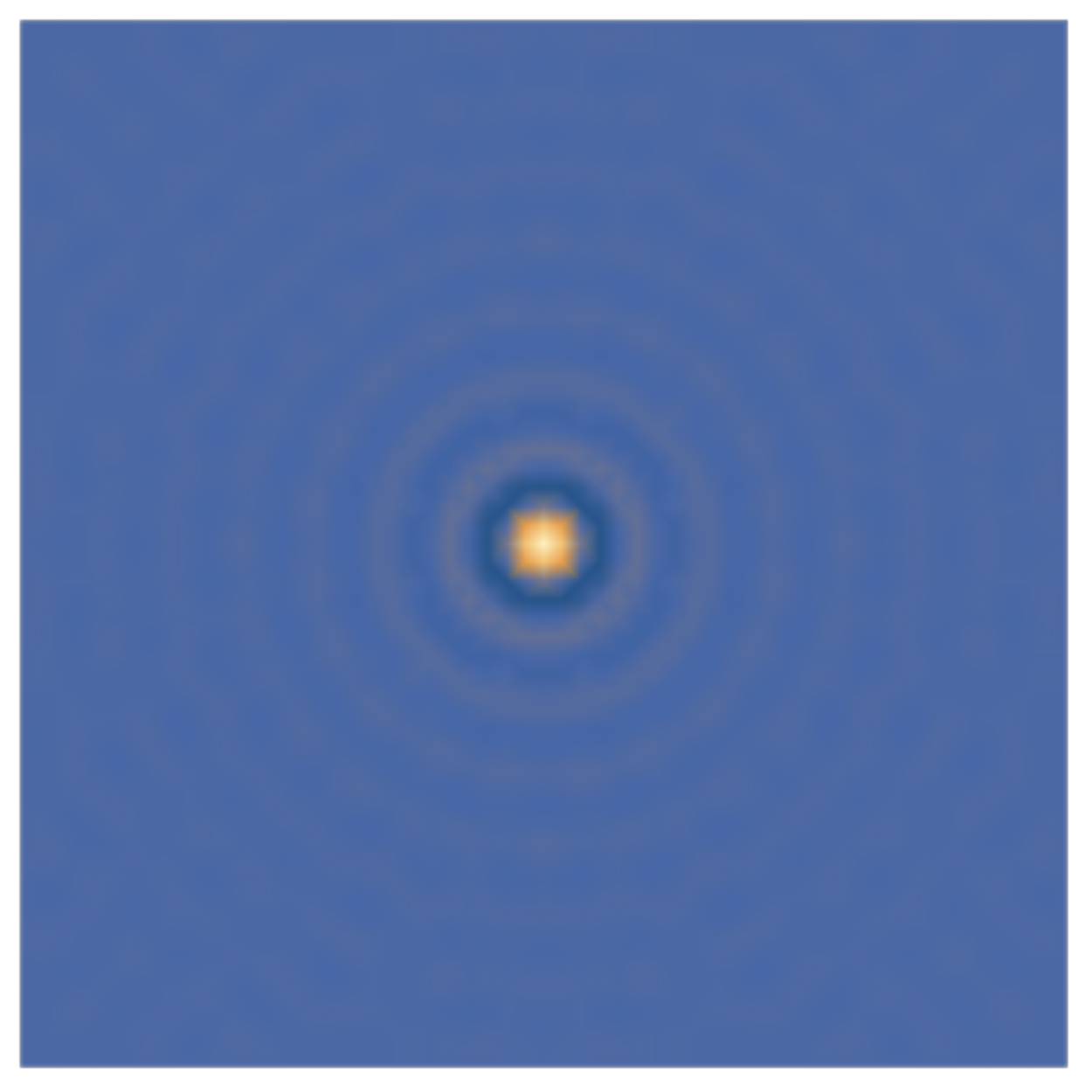}};
 \end{tikzpicture}
 \caption{Square $[0,\pi]^2$ with Dirichlet boundary conditions: $\am^{(675)}((1.3, 1.3),(x,y))$ (left) and $\am^{(675)}((\pi/2, \pi/2),(x,y))$ (right).}
 \end{figure}
 \end{center}
\vspace{-10pt}

\begin{definition} A basis $(\phi_{n})_{n=1}^{\infty}$ of Laplacian eigenfunctions does \underline{not} exhibit \emph{spooky action at a distance} if, for a constant $c$ and all $x,y \in M$ and all $n \in \mathbb{N}_{\geq 2}$
$$\left| \am^{(n)}(x,y) \right| \lesssim_{x,y} (\log{n})^c \sqrt{n}.$$
We say that it exhibits \emph{strong spooky action at a distance} if there exists $x \neq y$ with
$$ \limsup_{n \rightarrow \infty}  \frac{1}{n} \left| \am^{(n)}(x,y) \right| > 0.$$
\end{definition}

 On manifolds where eigenvalues have multiplicity (say $\mathbb{S}^2$), there are infinitely many choices of bases and some might exhibit spooky action at a distance while others do not. We will show (Theorem 1) that the standard basis on $\mathbb{S}^1$ exhibits spooky action but that a suitable randomization of the standard basis on $\mathbb{S}^1$ does not.
This definition is somewhat complementary to the random wave heuristic (which is purely local): the definition quantifies a certain lack of independence exhibited by Laplacian eigenfunctions when restricted to two different points (see \S 3.4 for a precise random model). A phenomenon related to spooky action at a distance, unexpected sign correlations at a distance, was already observed for one-dimensional Sturm-Liouville problems by Goncalves, Oliveira e Silva and the author \cite{stef}

\subsection{Spooky Action for Classical Examples} Most of the examples where computations can be done in closed form seem to exhibit \textit{strong} spooky action at a distance. This is interesting because it means that there are distinguished pairs of points $(x,y) \in M \times M$ such that a typical eigenfunction $\phi_k$ when evaluated in the point $x$ seems to care a great deal about its value in a far away point $y$. Equation \eqref{eq:up} implies that the set of such points is at most a set of measure 0 on $M \times M$.
\begin{theorem}
 $[0,1]$ (with either Dirichlet or Neumann boundary conditions) and the circle $\mathbb{S}^1$ exhibit \emph{strong} spooky action at a distance. A random basis of eigenfunctions on the circle $\mathbb{S}^1$ does not have spooky action at a distance (almost surely).
\end{theorem}
The list of examples could undoubtedly be continued and a better of understanding of $\am(x,y)$ for specific explicit manifolds $(M,g)$ might be of interest.  A recurring theme in the study of eigenfunction concentration is that most of
the examples of manifolds $(M,g)$ where explicit computations are possible do also exhibit eigenfunction concentration -- as has been
argued, the very fact that explicit computations are possible is due to the presence of additional structure in the manifold $(M,g)$ making
them atypical. We believe that spooky action at a distance is another such consequence and not to be expected on a `generic' manifold.
\vspace{-50pt}
 \begin{center}
 \begin{figure}[h!]
 \begin{tikzpicture}
 \node at (0,0) {\includegraphics[width=0.6\textwidth]{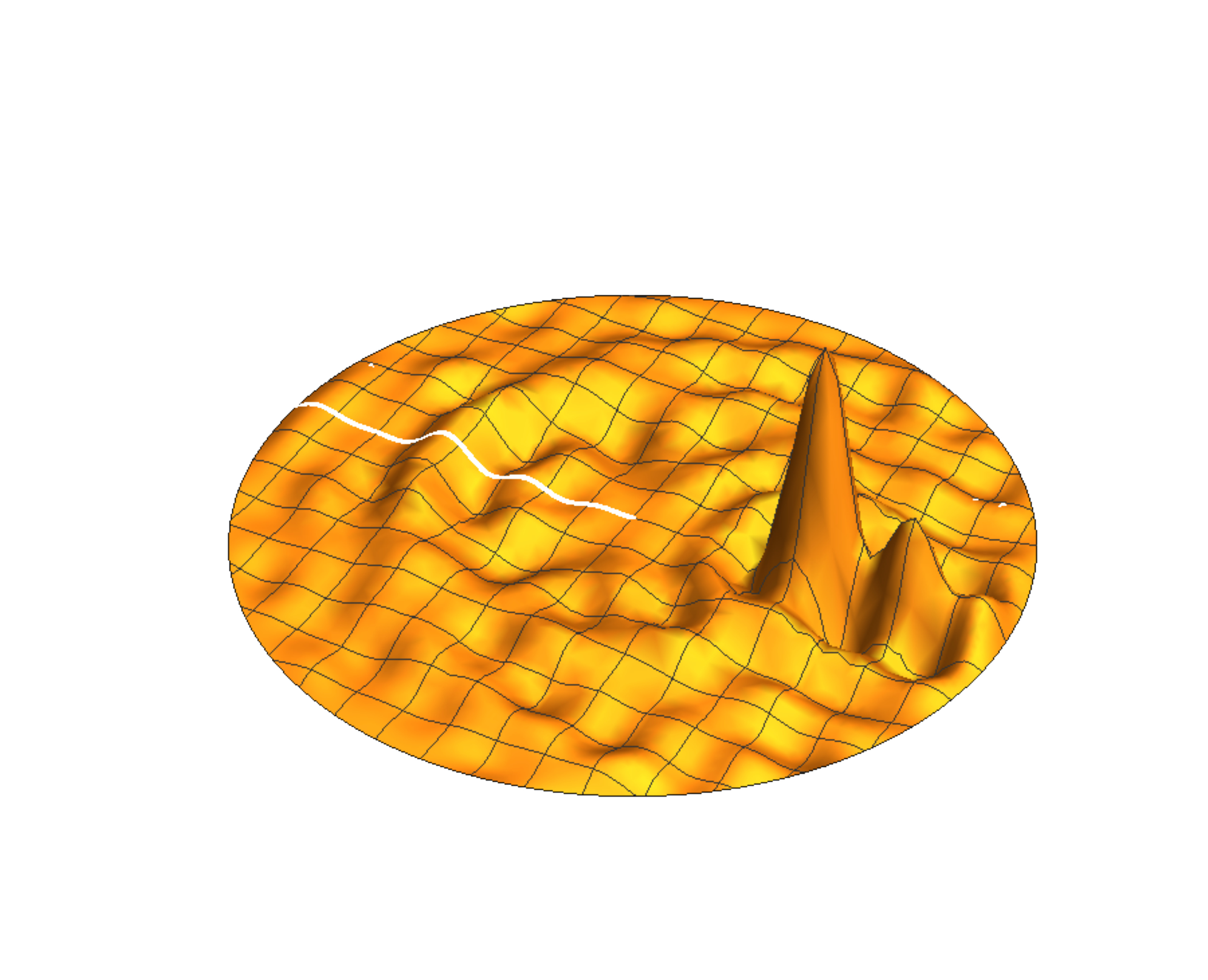}};
  \node at (5.5,0) {\includegraphics[width=0.35\textwidth]{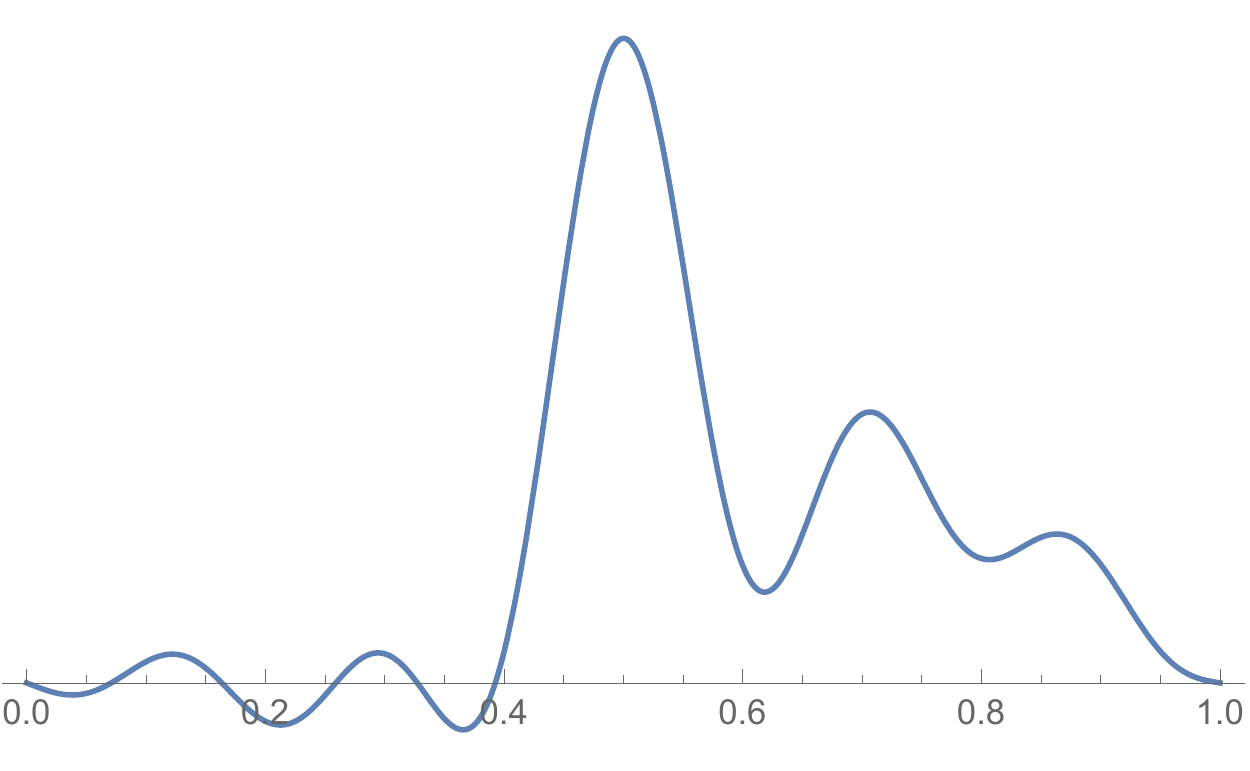}};
 \end{tikzpicture}
 \vspace{-35pt}
 \caption{Spooky Action at a Distance on the unit disk: $\am^{(120)}((0.5, 0),(x,y))$ (left) and $\am^{(120)}((0.5, 0),(x,0))$ (right).}
 \label{fig:bessel}
 \end{figure}
 \end{center}

\subsection{Interlude: $\am$ close to the diagonal.} We quickly present a problem that is interesting in itself (and will also play a role in the proof of Theorem 2). Suppose $x_0 \in M$ and $n \in \mathbb{N}$ are fixed. Where does the function (of $y$)
$$ \am^{(n)}(x_0,y) =   \sgn(\phi_1(x_0)) \phi_1(y) + \dots +  \sgn(\phi_n(x_0)) \phi_n(y) $$
assume its maximal value? Naturally, since the signs are all favorably aligned around $x_0$, one would assume that the function is presumably largest somewhere close to that point. There is no reason to assume that it should be exactly in $x_0$ but possibly somewhere nearby. Recall from \eqref{eq:univlower} that
$$  \am(x,x) = \sum_{k=1}^{n} |\phi_k(x)|  \gtrsim_{(M,g)}  n^{\frac{d+1}{2d}}$$
and that we expect, in absence of spooky action at a distance, that 
$$\am^{(n)}(x,y) \lesssim (\log n)^c \sqrt{n} \ll   n^{\frac{d+1}{2d}}$$
for points $x \neq y$ that are not close. We also note that
$$ \Delta_{y} \am^{(n)}(x,y) \big|_{y =x} = -\sum_{k=1}^{n} \lambda_k |\phi_k(x)|$$
which indicates that $\am(x,y)$, as a function of $y$, is locally very concave.
\begin{quote}
\textbf{Question 2.} When do we have, for all $x \in M$, and for all $n$ sufficiently large, an estimate of the type
\begin{equation} \label{eq:diag1}
 \max_{y \in M} \am^{(n)}(x,y) \leq c \am^{(n)}(x,x)?
\end{equation}
One might be tempted to conjecture a stronger statement that in `most' cases the constant
should tend to 1 as $n \rightarrow \infty$. When is
\begin{equation} \label{eq:diag2}
 \max_{y \in M} \am^{(n)}(x,y) \leq \left(1 + o_n(1)\right) \am^{(n)}(x,x)?
 \end{equation}
 When is it true that the point $x^*$ in which $\am^{(n)}(x,x^*) = \max_{z \in M} \am^{(n)}(x,z)$ assumes its maximum is not far away from $x$ and satisfies
 \begin{equation} \label{eq:diag3}
 d(x,x^*) \lesssim n^{-\frac{1}{d}}?
 \end{equation}
\end{quote}

  \begin{center}
 \begin{figure}[h!]
 \begin{tikzpicture}
 \node at (0,0) {\includegraphics[width=0.5\textwidth]{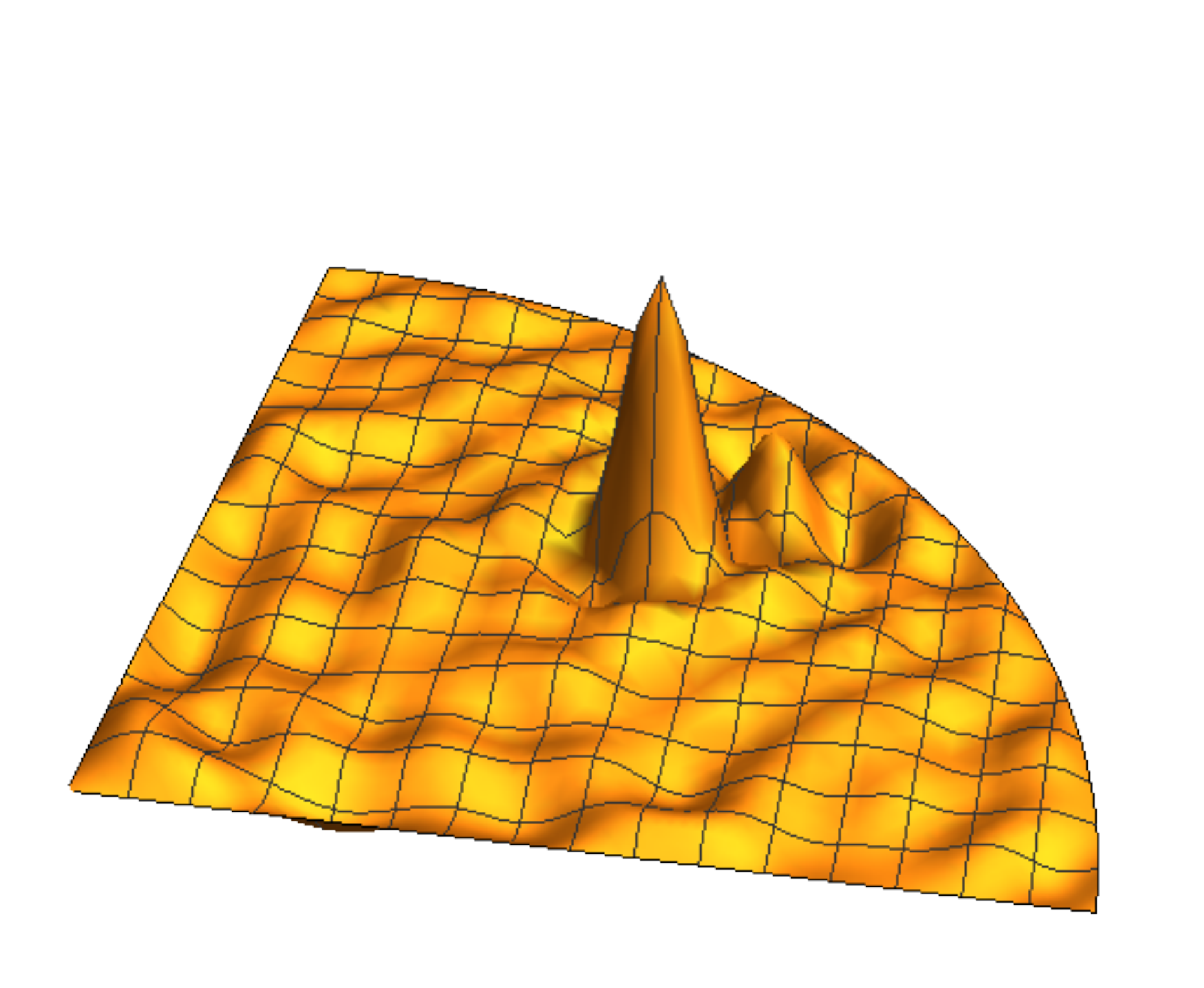}};
  \node at (6,0) {\includegraphics[width=0.38\textwidth]{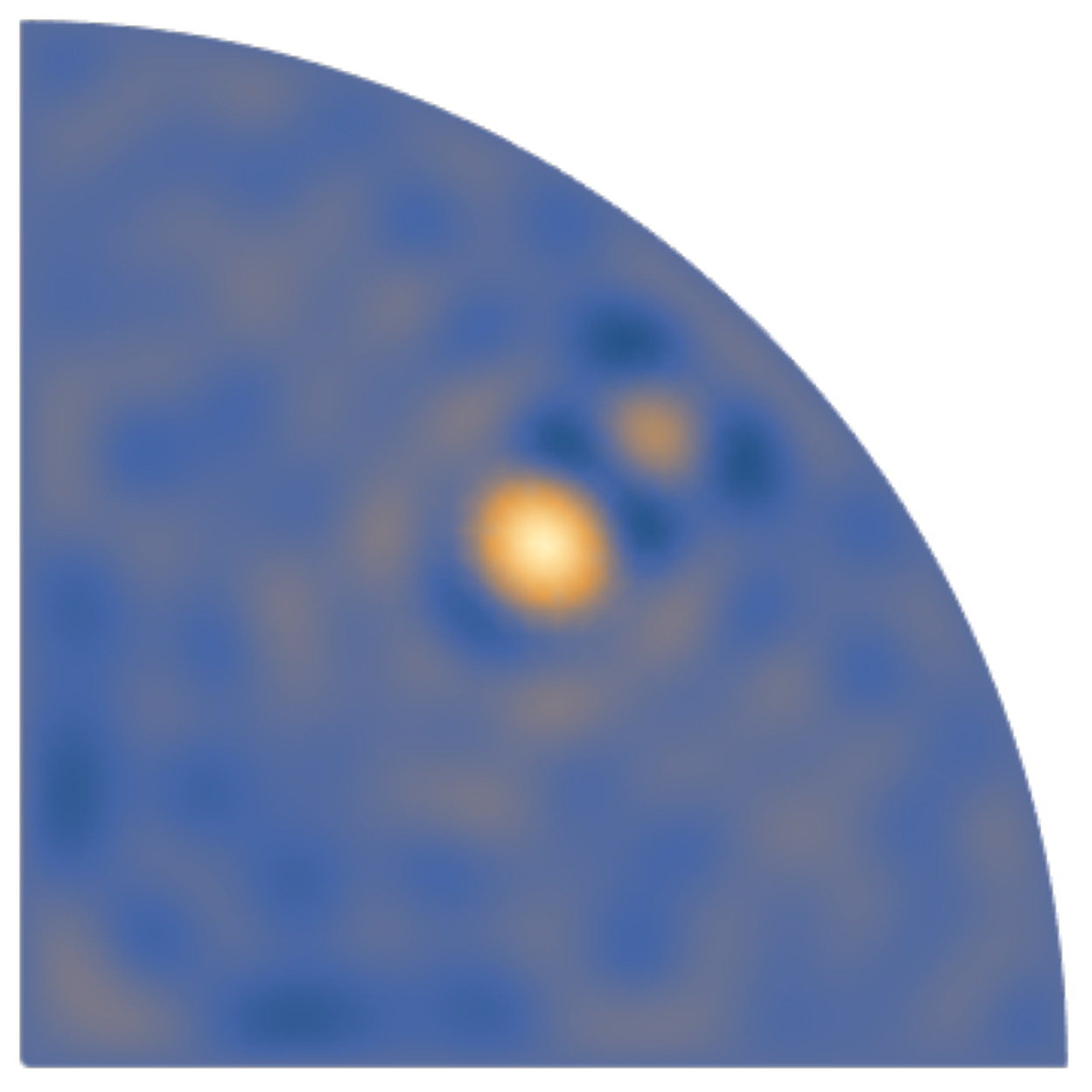}};
 \end{tikzpicture}
 \caption{Eigenfunctions on a quarter-disk with Dirichlet conditions. For the point $x = (0.5, 0.5)$ on the axis of symmetry, we see a formation of spooky action just behind $x$ itself $\am^{(189)}((0.5, 0.5), (x,y))$ (left and right).}
 \end{figure}
 \end{center}

We expect all three properties to be true for `generic' manifolds. \eqref{eq:diag1} might be true under more general conditions -- one has to be a bit careful with boundary effects: consider the orthogonal basis of functions $\phi_k(x) = \sin{(k x)}$ on the interval $[0,\pi]$ where $1 \leq k \leq n$. Then there is a $\sim n^{-1}$ neighborhood around 0 such that all eigenfunctions are positive and monotonically increasing in that neighborhood and we cannot hope for any such estimates -- needless to say, the example is non-generic in many ways. The same type of example also shows that \eqref{eq:diag2} is not necessarily always satisfied in a region bounded away from the boundary (see Proposition \ref{prop:sines}). In contrast, the property $\eqref{eq:diag3}$ might be the most robust of the three: it failing would mean that one deals with \textit{very} spooky action at a distance.
\begin{proposition} \label{prop:sines}
Consider the orthogonal basis of functions $\phi_k(x) = \sin{(k x)}$ on the interval $[0,\pi]$ where $1 \leq k \leq n$. Then, for $n$ sufficiently large, there exists $x_n \sim \pi/2 + n^{-10}$
and $y_n \sim x_n + 1.17 n^{-1}$ such that
$$ \am^{(n)}(x_n, y_n)   \sim 1.308 \cdot \am^{(n)}(x_n,x_n).$$
\end{proposition}
It could be of interest to study \eqref{eq:diag1}, \eqref{eq:diag2} and \eqref{eq:diag3} on other manifolds where eigenfunctions can be explicitly computed. We reiterate that manifolds $(M,g)$ on which explicit computations with eigenfunctions are possible are those endowed with additional structure which are unlikely to behave like `generic' manifolds. A nice exception is the sphere $\mathbb{S}^d$ with a randomly chosen basis; random bases on the sphere seem to behave, with regards to computable statistics, much as one would expect a `generic' basis of eigenfunctions on a generic manifold to behave.

\subsection{$L^{\infty}-$growth: the main idea} The purpose of this section is to illustrate the connection between $\am^{(n)}(x,y)$ and eigenfunction growth; we will keep the presentation informal for the sake of clarity of exposition -- a rigorous formulation of these ideas will be presented in \S 2.6.

\subsubsection{The Idea.} Suppose that we are given the first $n$ eigenfunctions $\phi_1, \dots, \phi_n$ and that we try to understand whether $\phi_{n+1}$ can be large in a point $x_0 \in M$ where $x_0$ is the point chosen so that $\phi_{n+1}$ (after possibly flipping its sign) assumes its global maximum in $x_0$, i.e. $\phi_{n+1}(x_0) = \| \phi_{n+1}\|_{L^{\infty}(M)}$. 
We first note that because of orthogonality of eigenfunctions
$$ \int_{M} \am^{(n)}(x_0,y) \phi_{n+1}(y) dy = 0.$$
At the same time, recalling the lower bound  \eqref{eq:easylower}, we have
 $$ \am^{(n)}(x_0,x_0) \geq \frac{ n - \mathcal{O}(n^{\frac{d-1}{d}})}{\max_{1 \leq k \leq n} \| \phi_k\|_{L^{\infty}}}.$$
In the generic setting, where the first $n$ eigenfunctions do not exhibit a lot of concentration, we expect $ \am^{}(x_0,x_0)$ to not be much smaller than $n$ (perhaps by some multiplicative logarithmic factors). 
At the same time, since it is a linear combination of eigenfunctions the largest of which oscillates at frequency $\lambda_n \sim n^{2/d}$, we expect all these functions to be locally constant at the wavelength $1/\sqrt{\lambda_n}$ and
$$\am^{(n)}(x_0, y) \sim \am^{(n)}(x_0, x_0) \qquad \mbox{for} \qquad d(x_0, y) \leq  c\cdot \lambda_n^{-\frac{1}{2}} \sim n^{-\frac{1}{d}}.$$
However, the same is also true for $\phi_{n+1}$. We note that the associated eigenvalue $\lambda_{n+1}$ is not much larger than $\lambda_n$ (in particular, $\lambda_{n+1}/\lambda_n \rightarrow 1$) implying that it is also locally constant (interpreted in a suitable sense) at scale $n^{-1/d}$. 
Using orthogonality, we arrive at
\begin{align*}
  \int_{M \setminus B(x_0, n^{-1/d})} \am^{(n)}(x_0,y) \phi_{n+1}(y) dy = -\int_{B(x_0, n^{-1/d})} \am^{(n)}(x_0,y) \phi_{n+1}(y) dy 
\end{align*}
Since everything is nearly constant at scale $\sim n^{-1/d}$, we end up with an estimate for the right-hand side which is
of the form
\begin{align*}
\int_{B(x_0, n^{-1/d})} \am^{(n)}(x_0,y) \phi_{n+1}(y) dy \sim \frac{ \am^{(n)}(x_0, x_0)}{n} \cdot \| \phi_{n+1}\|_{L^{\infty}}.
\end{align*}
This then implies that, up to logarithmic factors, 
$$  \| \phi_{n+1}\|_{L^{\infty}} \sim  \left| \int_{M \setminus B(x_0, n^{-1})} \am^{(n)}(x_0,y) \phi_{n+1}(y) dy \right|.$$
 We immediately see that this integral being large
is actually somewhat curious: it indicates that the eigenfunction $\phi_{n+1}(y)$ is strongly correlated with a suitable linear combination
of the first $n$ eigenfunctions on most of the manifold minus a small ball, $M \setminus B(x_0, n^{-1/d})$, which would be extremely interesting and unexpected.
\vspace{-20pt}
  \begin{center}
 \begin{figure}[h!]
 \begin{tikzpicture}
 \node at (0,0) {\includegraphics[width=0.5\textwidth]{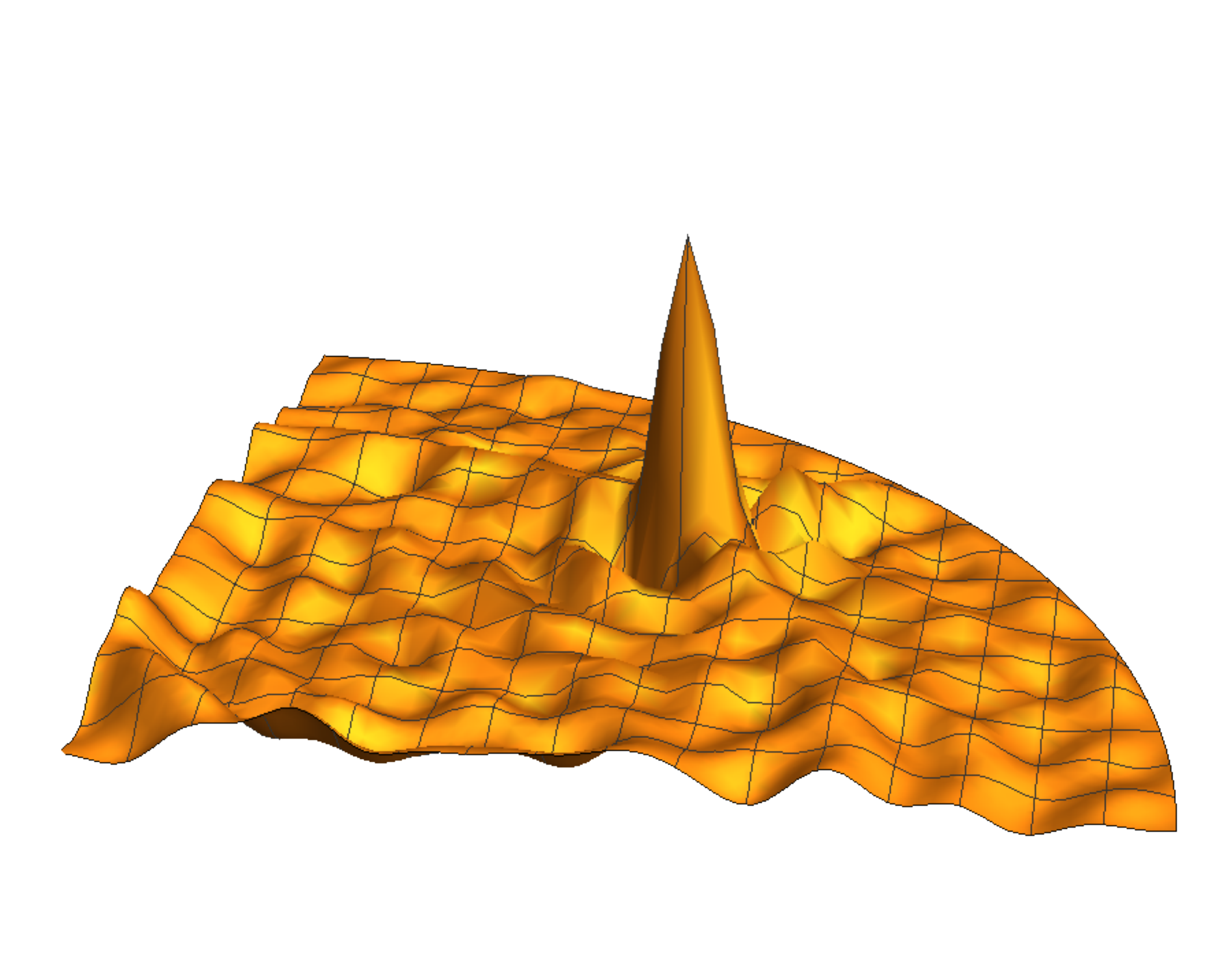}};
  \node at (6,0) {\includegraphics[width=0.35\textwidth]{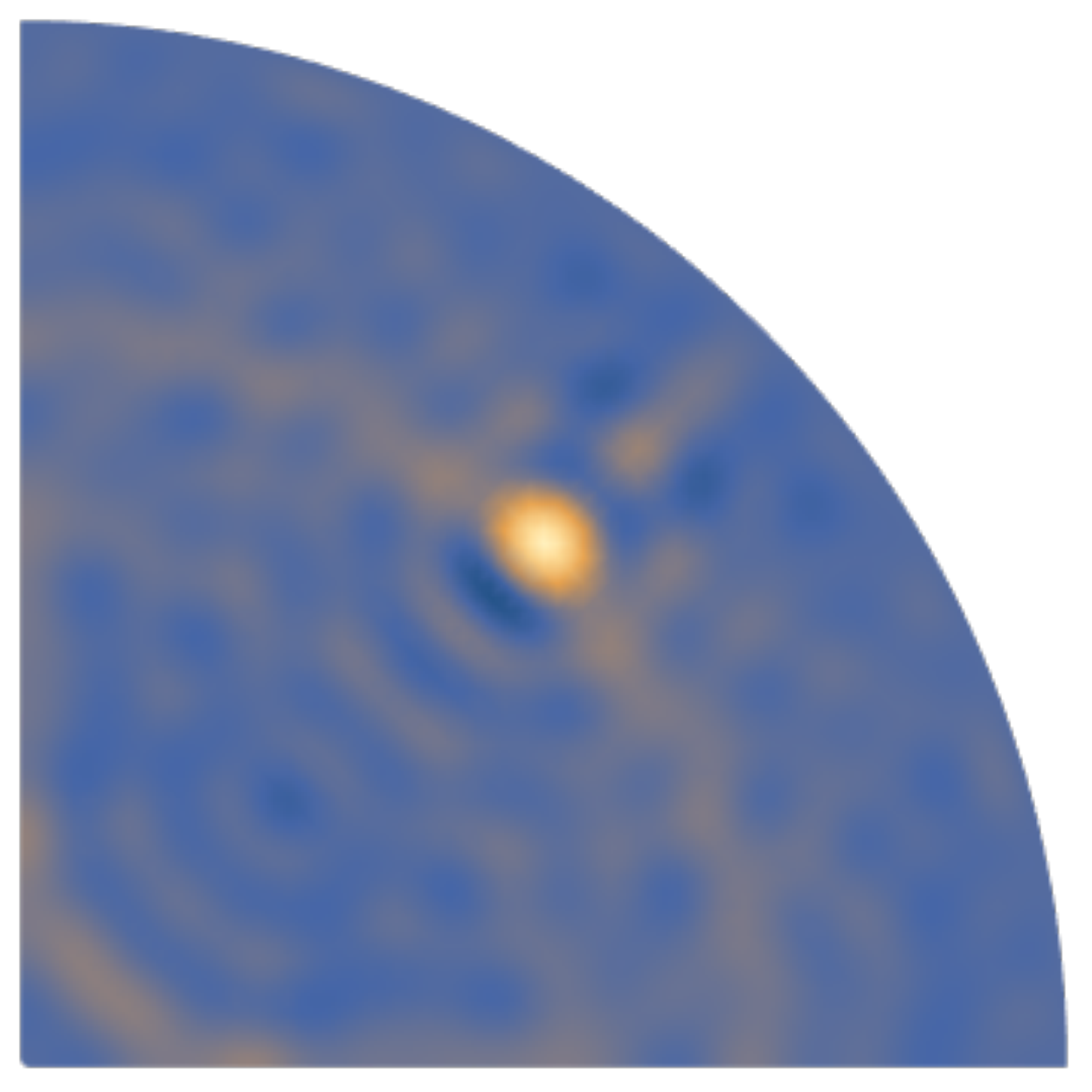}};
 \end{tikzpicture}
 \caption{Eigenfunctions on a quarter-disk with Neumann conditions: $\am^{(181)}((0.5, 0.5), (x,y))$ (left and right). The picture hints at spooky action at a distance along an entire curve.}
 \end{figure}
 \end{center}
\vspace{-20pt}
\subsubsection{Some Easy Consequences.} 
 A basic application of Cauchy-Schwarz, together with \eqref{eq:up} and Weyl's theorem implies
$$   \left| \int_{M \setminus B(x_0, n^{-1})} \am^{(n)}(x_0,y) \phi_{n+1}(y) dy \right| \lesssim \lambda_{n}^{d/4}.$$
This is weaker than H\"ormander's bound by a factor of $\lambda_n^{1/4}$. If we assume that the eigenfunctions do \textit{not} exhibit
spooky action at a distance, then we can get, up to logarithmic factors, a stronger bound
$$   \left| \int_{M \setminus B(x_0, n^{-1})} \am^{(n)}(x_0,y) \phi_{n+1}(y) dy \right| \lesssim \lambda_{n}^{d/4} \cdot \| \phi_{n+1} \|_{L^1}.$$
Both these inequalities are, in some sense, missing the point since them being even close to sharp would imply
strong algebraic connection between $\am$ and $\phi_{n+1}$ which is much more interesting and one way of phrasing the main insight behind this argument: eigenfunction growth is the consequence of unexpected correlations.
\subsubsection{Random Wave Model.} Finally, we return to the random wave model and use it to try and understand the size of the integral.  If
$\phi_{n+1}(y)$ behaves like a random wave, then we should think of $\phi_{n+1}$ as being locally described by random Gaussians at scale $n^{-1/d}$. This means the integral should be roughly, using again \eqref{eq:up}, behave as a random variable
\begin{align*}
  \int_{M \setminus B(x, c_2 n^{-1})} \am^{(n)}(x,y) \phi_{n+1}(y) dy &\sim \am^{(n)}(x,y_1)\cdot \frac{\pm 1}{n} + \dots +  \am^{(n)}(x,y_n)\cdot \frac{\pm 1}{n} \\
&\sim \pm \sqrt{ \frac{ \am^{(n)}(x,y_1)^2}{n^2} + \dots + \frac{ \am^{(n)}(x,y_n)^2}{n^2}  }  \sim \pm 1.   \end{align*}
Thus, whenever the random wave heuristic applies, the integral will typically be at scale $\sim 1$ implying that the eigenfunctions do not
concentrate (up to logarithmic factors). This may appear to be circular reasoning (since the random wave model itself has lack of concentration
built into it) -- however, note that in this argument the random wave model is only used in integrated form (and not pointwise); it also illustrates how this approach
is naturally aligned with the random wave model.

\subsubsection{Summary.}  We can now, informally, summarize these ideas as follows: in order for $\phi_{n+1}(x)$ to be very large, we either have
\begin{enumerate}
\item that $\am^{}(x,x)$ is unexpectedly small (which, in this case, would mean something like $\am^{(n)}(x,x) \ll n$ by more than multiplicative logarithmic factors) 
\item or $\phi_{n+1}(y)$ is correlated with $\am^{(n)}(x,y)$ on the manifold minus a ball 
\item or both.
\end{enumerate}
Note that $\am^{(n)}(x,x)$ being unexpectedly small requires many of the first $n$ eigenfunctions to be unexpectedly large in $x$ (this follows from the proof of Proposition 2). We recall that
 $$ \am(x,y) = \varepsilon_1 \phi_1(y) + \varepsilon_2 \phi_2(y) + \dots + \varepsilon_n \phi_n(y) \qquad \mbox{where} ~\varepsilon_i \in \left\{-1,0,1\right\}$$
 and that it would be quite unexpected to have this be in any way structurally aligned with $\phi_{n+1}(y)$ unless there was some amount of algebraic structure present. There is a bootstrapping aspect to the argument: (1) can, in some sense, not be the origin of eigenfunction growth since it relies on concentration already being present in earlier eigenfunctions; \textit{the driving force behind eigenfunction concentration comes from unexpected correlations}. Conversely, if we have spooky action a distance, say $\am^{}(x, y) \sim \am(x,x)$ for $x \neq y$, then the entire approach cannot suitably exclude $\phi_{n+1}(x)$ and $\phi_{n+1}(y)$ being large simultaneously since these contributions could then cancel in the integral.\\
 
 We conclude by noting that there might be more than one way of formalizing this idea. Theorem 2, presented in the next section, is a fairly verbatim way of capturing the essence of the argument. Note, however, that nothing in the argument was particularly sensitive about flipping a few extra signs and this approach leads to an entirely different way of thinking about the approach (see \S 3.3 for details).

\subsection{$L^{\infty}-$growth: Theorem 2}

We will argue that $\am(x,y)$, spooky action at a distance and $L^{\infty}-$growth of eigenfunctions are all connected. We first state the result and then
discuss it in greater detail.

\begin{theorem} Suppose $(M,g)$, normalized $\vol(M) = 1$, has a basis of Laplacian eigenfunctions $(\phi_n)_{n=1}^{\infty}$ and suppose the eigenfunction $\phi_{n+1}$ assumes its maximum in $z \in M$. There exists a constant $c_{\kappa}$ only depending on 
$$ \kappa = \frac{ \max_{w \in M} \am^{(n)}(z,w)}{\am^{(n)}(z,z)} \geq 1$$
such that for all $t \leq 1/(4 \kappa \lambda_{n+1})$
$$  \phi_{n+1}(z) \leq c_{\kappa} \frac{2n}{\am(z,z)} \left|  \int_{M} \left(1 - \frac{p_t(z,y)}{\max_{w \in M} p_t(z,w)} \right) \am(z,y) \phi_{n+1}(y) dy \right|.$$
\end{theorem}

\textbf{Remarks.}
\begin{enumerate}
\item As discussed in \S 2.4, we expect that typically $\kappa \sim 1$. In particular, we expect to be able to use Theorem 2 with $t \sim \lambda_{n}^{-1}$.
\item On the length scale $t \sim \lambda_n^{-1}$, the heat kernel $p_t(x,y)$ is localized at scale $\sqrt{t} \sim 1/\sqrt{\lambda_n}$ which 
is comparable to the wavelength. We then expect
$$  \left(1 - \frac{p_t(z,y)}{\max_{w \in M} p_t(z,w)} \right) \sim \begin{cases} 0 \qquad &\mbox{if}~d(z,y) \leq 1/\sqrt{\lambda_n} \\ 
1 \qquad &\mbox{otherwise.} \end{cases}$$
The weight can be understood as effectively filtering out the behavior of $\am(z,y) \phi_{n+1}(y)$ for $y$ at the scale of one wavelength around $z$.
\item As discussed in \S 2.5, there are two ways how $\phi_{n+1}(z)$ can be large: either $\am(z,z) \ll n$ or the integral is large: if $\am(z,z) \ll n$, then
many previous eigenfunctions had to be strongly localized in $z$. If the integral is large, then this means global negative correlation of $\am(z,y)$ and $\phi_{n+1}(y)$ one
wavelength from the point $z$ where the eigenfunction assumes its maximum.
\item A natural example is the classical zonal spherical harmonic on $\mathbb{S}^2$ (explained in greater detail in \S 5.3). There we have $\am(z,z) \sim n^{3/4} \ll n$, the integral is actually relatively small $\sim 1$ and $\| \phi_{n+1}\|_{L^{\infty}} \sim n^{1/4}$.
\end{enumerate}

\section{Comments and Remarks}
\subsection{Special Function Identities.} One interesting byproduct of the approach is that on any manifold $(M,g)$ 
with eigenfunction growth, we either have spooky action at a distance (which is interesting in itself) or a strong correlation between $\am^{(n)}(x,y)$ and $\phi_{n+1}$ over most of the manifold (which is also interesting in itself). We illustrate this for two different examples.

\subsubsection{The Dirichlet kernel.} Our first example may at first glance seem somewhat paradoxical: we consider the standard basis on $\mathbb{S}^1$. Naturally, there is no eigenfunction growth on one-dimensional manifolds. Nonetheless, as also explained by Theorem 1, the standard basis on $\mathbb{S}^1$ does exhibit characteristics typical of manifolds with eigenfunction growth (for example spooky action at a distance). We now consider the standard Fourier basis on $\mathbb{S}^1$
$$ \am^{(2n)}(x,y) = \sum_{k=1}^{n} \sgn(\sin(kx)) \sin(ky) + \sum_{k=1}^{n} \sgn(\cos(kx)) \cos(ky) $$
and try to understand what it says about the next eigenfunction $\cos{((n+1)y)}$ which, suitably interpreted, has `maximal eigenfunction growth' in the origin (somewhat paradoxically so). We simplify
$$ \am^{(2n)}(0,y) =\sum_{k=1}^{n} \cos(ky) = \csc \left(\frac{x}{2}\right) \sin \left(\frac{n x}{2}\right) \cos \left(\frac{1}{2}
   (n+1) x\right).$$
   This expression is orthogonal to $\cos{((n+1)y)}$. However, as predicted by Theorem 2, once we remove the origin, there is indeed a naturally appearing global negative correlation between $ \am^{(2n)}(0,y)$ and $\cos{((n+1)y)}$ (see Fig. \ref{fig:diri}). In stark contrast, nothing like this happens if we randomize the basis of eigenfunctions (see Fig. \ref{fig:destrospooky}).
  \begin{center}
 \begin{figure}[h!]
 \begin{tikzpicture}
 \node at (0,0) {\includegraphics[width=0.45\textwidth]{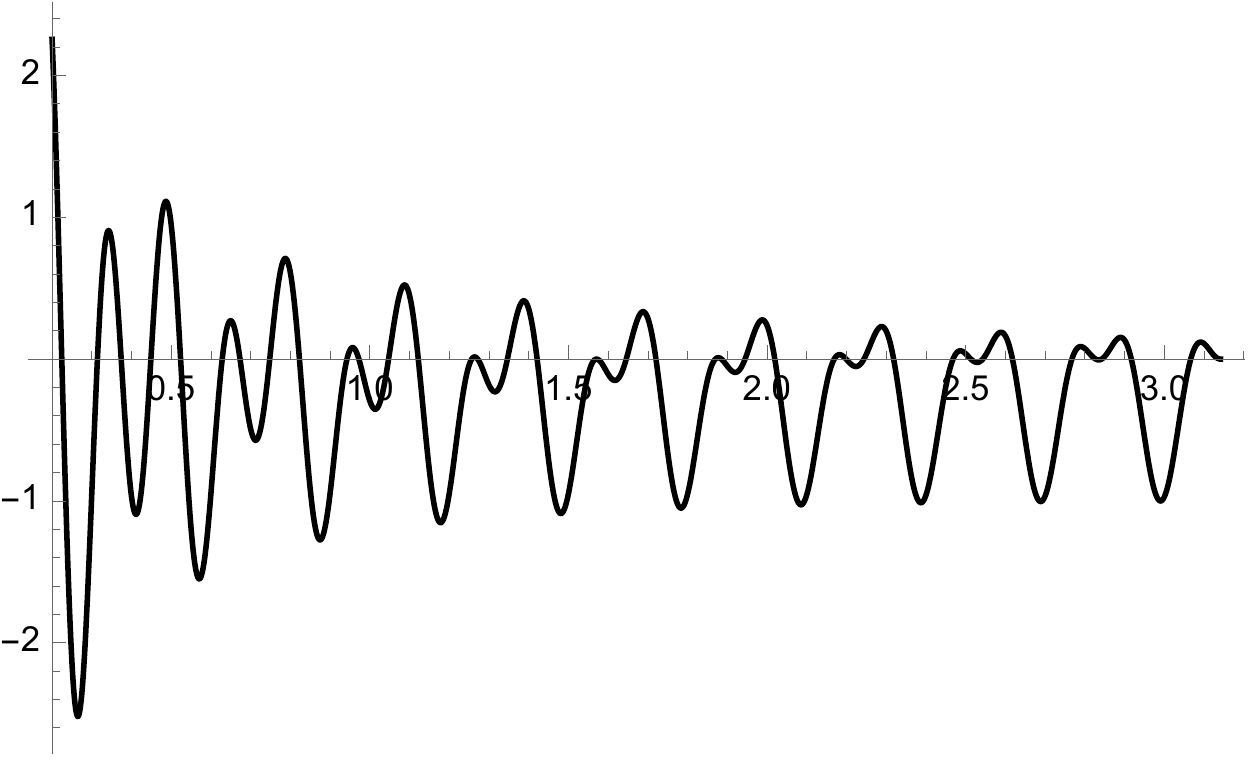}};
  \node at (6,0) {\includegraphics[width=0.45\textwidth]{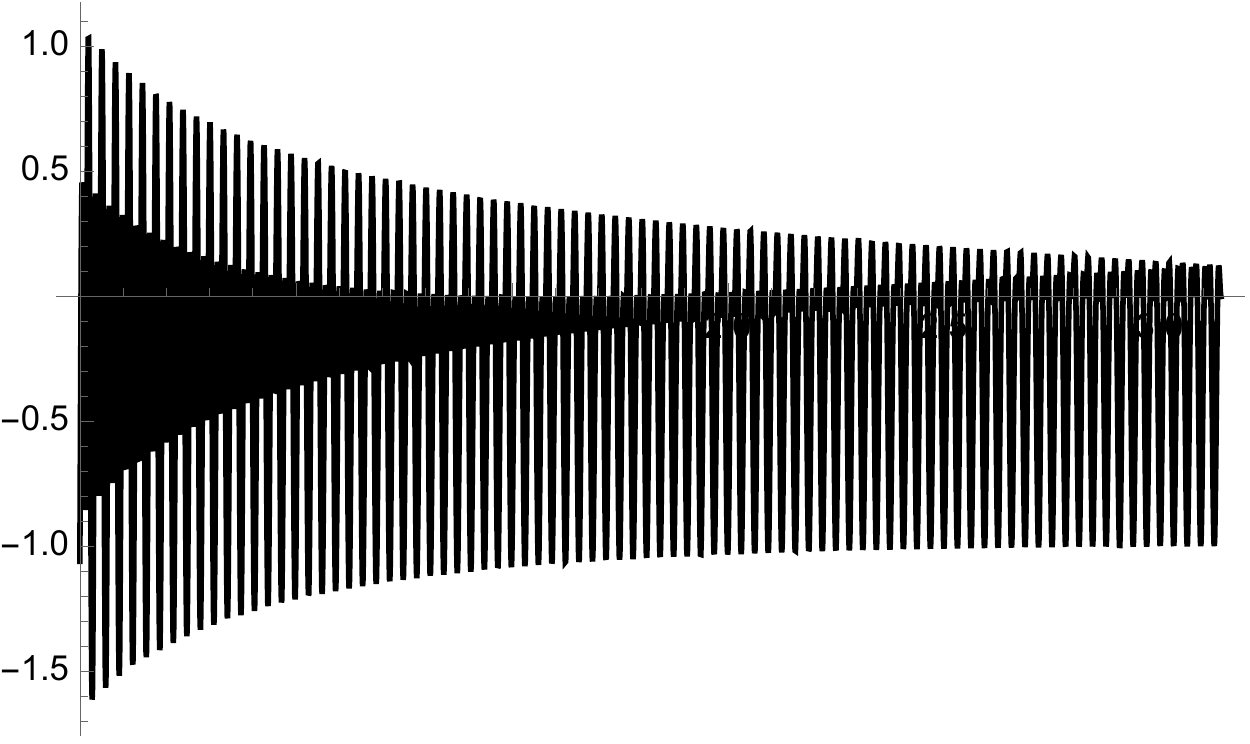}};
 \end{tikzpicture}
 \caption{ $\am^{(20)}(0, y) \cdot \cos{(21y)}$ (left) and $\am^{(201)}(0, y)\cos{(201y)}$ (right) away from the origin: global negative correlation. }
 \label{fig:diri}
 \end{figure}
 \end{center}

\subsubsection{Zonal spherical harmonics} Consider now the zonal spherical harmonics on $\mathbb{S}^2$ (eigenfunctions on $\mathbb{S}^2$ exhibiting maximal possible eigenfunction growth). These are radial and maximal at north and south pole. Since simultaneously, all other eigenfunctions vanish in the north/south pole, we can simple write everything as a function of $\theta$ and all arising quantities become functions of one variable. Defining $\phi_k:[0,\pi] \rightarrow \mathbb{R}$ via
$$ \phi_k(\theta) = \sqrt{k + \frac12} \cdot P_k(\cos{\theta})$$
we have the orthogonality relation
$$\int_0^{\pi} \phi_k(\theta) \phi_{\ell}(\theta) \sin{(\theta)} d\theta = \delta_{k \ell},$$
where $\sin{(\theta)}$ is the Jacobian from the area on the sphere.
We refer to \S 5.4 for more details.
Simultaneously, because $\phi_k(0) > 0$, we have
$$ \am(\mbox{north pole}, \theta) = \sum_{k=1}^{n}  \sqrt{k + \frac12} \cdot P_k(\cos{\theta}).$$
We observe that, $\am(0,\theta)$ and $\phi_{k+1}(\theta)$ have strong correlation in the origin $\theta =0$ followed by a global negative correlation away from the origin (see Fig. \ref{fig:heg})
In much the same spirit, we observe that spooky action at a distance numerically observed on the disk (see Fig. \ref{fig:bessel}) hints at a related phenomenon for sums of Bessel functions. Similarly, one might be able to revisit the examples of Bourgain \cite{bour2} and Rudnick \& Sarnak \cite{rudnick}.
We have not pursued this direction further but it is clear that on any manifold $(M,g)$ where explicit computations with eigenfunctions are possible, one might hope to be able to recover statements about special functions of this sort.
  \begin{center}
 \begin{figure}[h!]
 \begin{tikzpicture}
 \node at (0,0) {\includegraphics[width=0.45\textwidth]{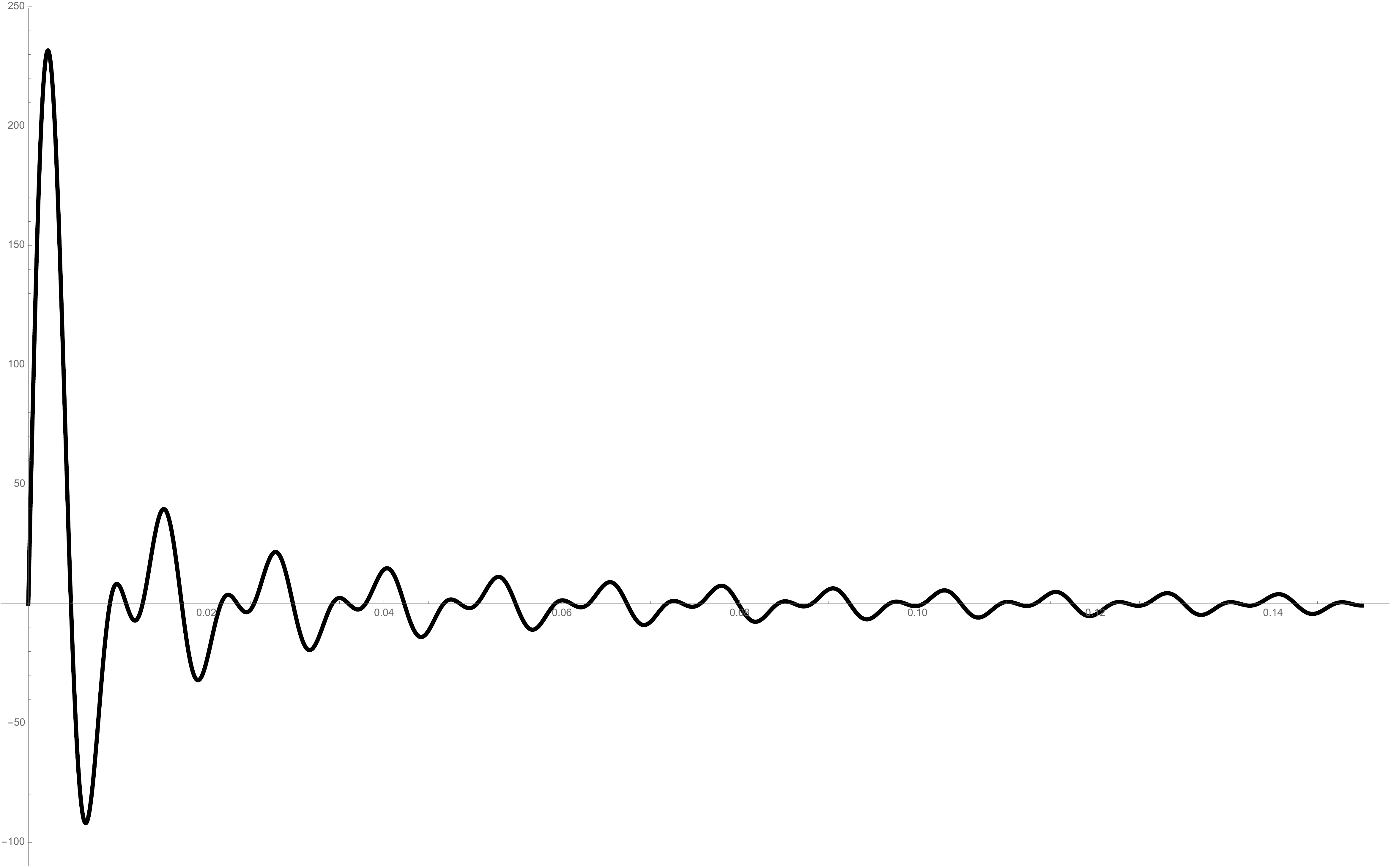}};
  \node at (6,0) {\includegraphics[width=0.45\textwidth]{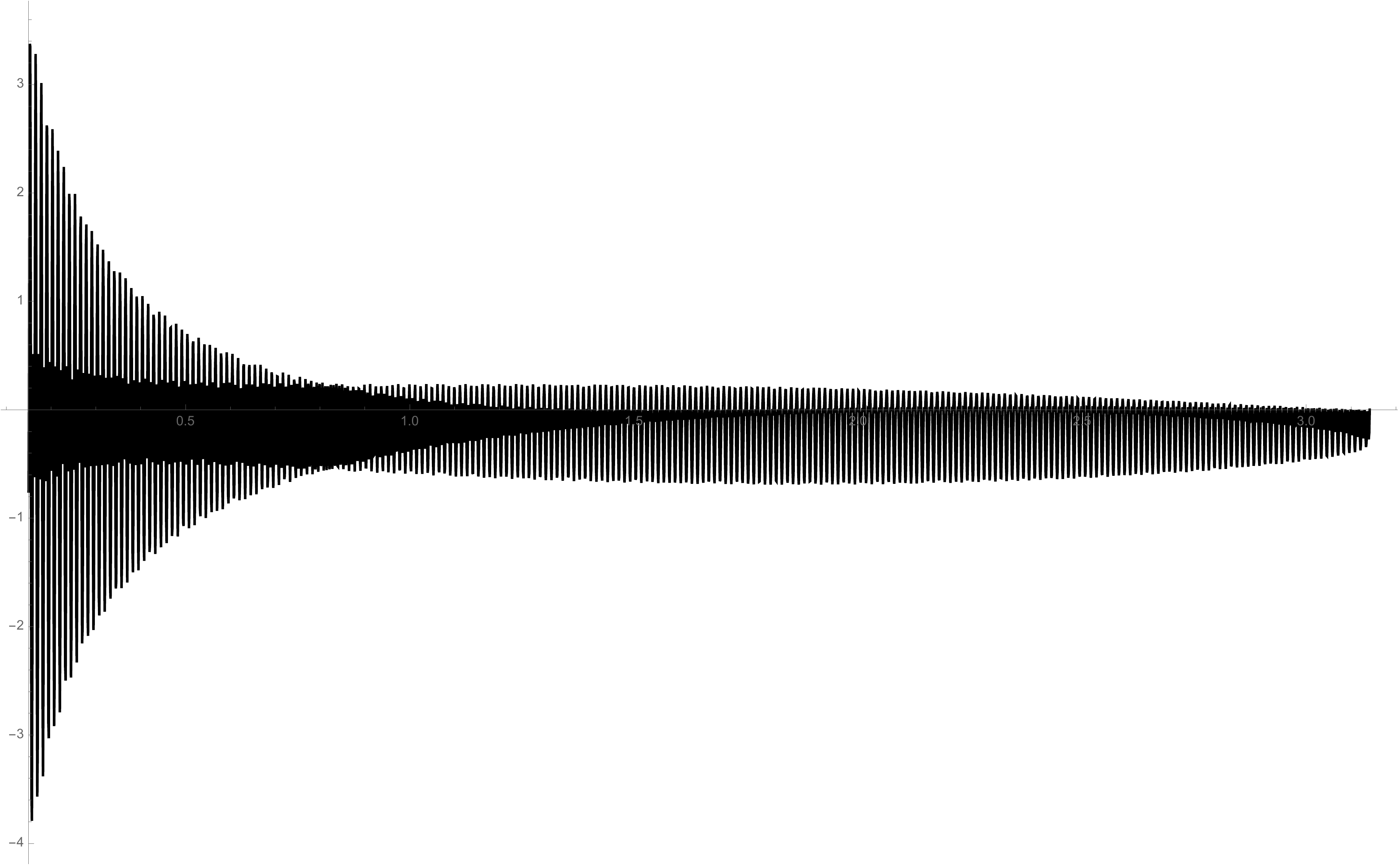}};
 \end{tikzpicture}
 \caption{A curious correlation between $\am^{(500)}(0, \theta)$ and $\phi_{501}(\theta)$. $\am^{(500)}(0, \theta)\phi_{501}(\theta)\sin{(\theta)}$ on $[0, 0.15]$ (left) and on $[0.15, \pi]$ (right). }
 \label{fig:heg}
 \end{figure}
 \end{center}

\subsection{General orthonormal bases.}
We also emphasize that the definition
$$ \am^{(n)}(x,y) = \sum_{k=1}^n \sgn(\phi_k(x)) \phi_k(y)$$
is not restricted to eigenfunctions but might be of interest for general families of functions. We illustrate this using
the Hermite functions on $\mathbb{R}$. The Hermite functions $(\psi_n)_{n=0}^{\infty}$
given by
$$ \psi_n(x) = \frac{1}{ \sqrt{\pi}}\frac{1}{2^{n/2} \sqrt{n!}} e^{-\frac{x^2}{2}} H_n(x),$$
where $H_n$ is the $n-$th Hermite polynomial.
They are an orthonormal family of functions that form a basis of $L^2(\mathbb{R})$. They can be naturally interpreted as eigenfunctions of the
operator $-\Delta + x^2$ but also make sense as an interesting family of functions in their own rights (for example for diagonalizing the Fourier transform).

  \begin{center}
 \begin{figure}[h!]
 \begin{tikzpicture}
 \node at (0,0) {\includegraphics[width=0.5\textwidth]{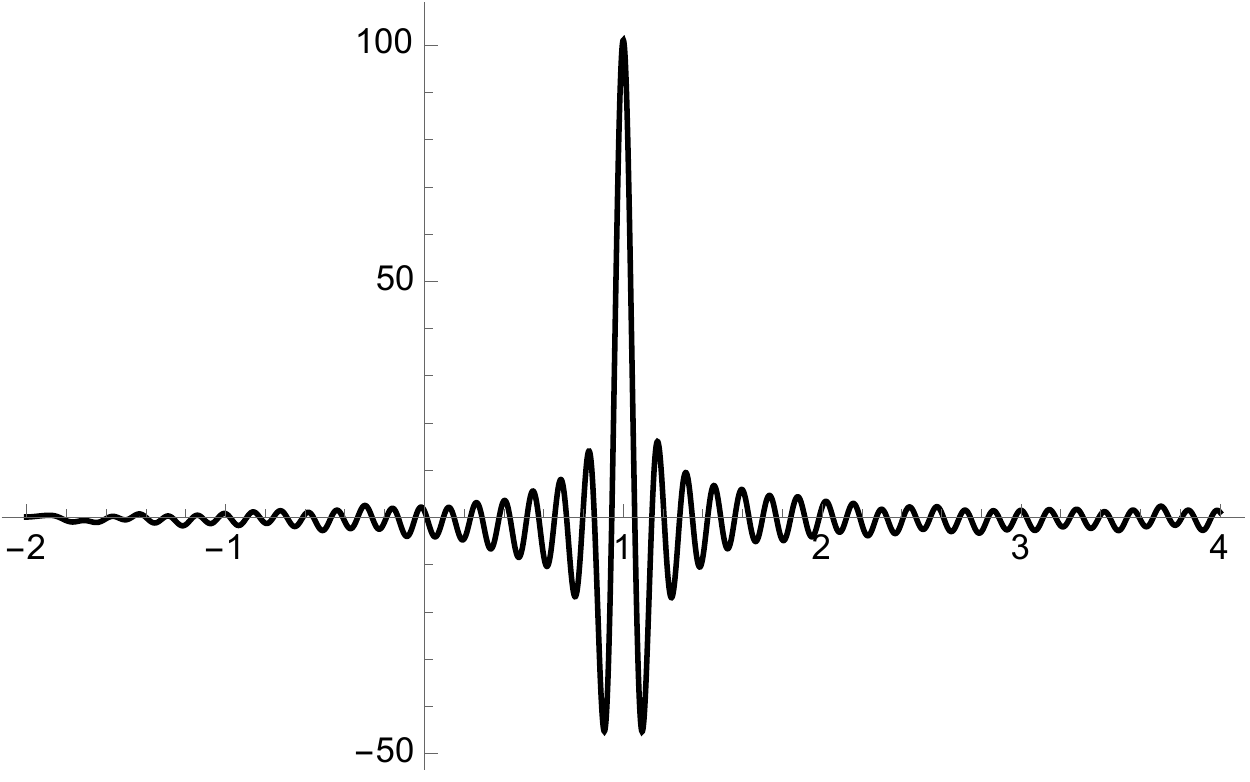}};
 \end{tikzpicture}
 \caption{ Hermite functions: $\am^{(1000)}(1,y)$ appears to be localized.}
  \label{fig:hermite}
 \end{figure}
 \end{center}

Asymptotically, for $x$ fixed, as $n \rightarrow \infty$, to leading order 
$$ \psi_n(x) \sim \frac{2^{1/4}}{\sqrt{\pi}} \frac{1}{n^{1/4}} \cos{\left(x \sqrt{2n} - \frac{n \pi}{2} \right)}.$$
Using this approximation as a suitably proxy, we see that it appears as if the Hermite functions might be
a candidate for a system of orthonormal functions with no spooky action at a distance (see Fig. \ref{fig:hermite}).
Given the large number of interesting orthogonal sequences that one might consider, we have not pursued this further
at this point but believe it to be a potentially interesting avenue for further research.

\subsection{More sign flips and random fields}
The idea outlined in \S 2.5 can be very concisely summarized as follows: we have, by orthogonality,
  $$\int_{M \setminus B(x_0, n^{-1/d})} \am^{(n)}(x_0,y) \phi_{n+1}(y) dy = -\int_{B(x_0, n^{-1/d})} \am^{(n)}(x_0,y) \phi_{n+1}(y) dy.$$
  However, we also have that if $\phi_{n+1}(x_0) = \| \phi_{n+1}\|_{L^{\infty}},$ then
  $$ \int_{M \setminus B(x_0, n^{-1/d})} \am^{(n)}(x_0,y) \phi_{n+1}(y) dy \sim \frac{\am^{(n)}(x_0, x_0)}{n} \| \phi_{n+1}\|_{L^{\infty}}$$
  which forces $\am^{}(x_0, y)$ and $\phi_{n+1}(y)$ to be nontrivially correlated across $(M,g)$. 
Suppose now for a moment that $\am^{}(x_0, x_0) \sim n$. The local Weyl law
$$ \sum_{k=1}^{n} \phi_k(x_0)^2 = n + o(n)$$
 implies that at least a constant proportion of the first $n$ eigenfunctions satisfy $\left| \phi_k(x_0)\right| \leq 2$. If we were to
take such an eigenfunction and flipped its sign in the definition of $\am^{}(x,y)$, not too much in the above argument would 
change since $\am^{}(x_0, x_0) \sim n$ is still valid. Defining a permutation $\pi: \left\{1,2,\dots, n\right\} \rightarrow \left\{1,2, \dots, n\right\}$ which orders the eigenfunctions in the sense that
$$ \left| \phi_{\pi(1)}(x_0)\right| \leq  \left| \phi_{\pi(2)}(x_0)\right| \leq  \dots \leq \left| \phi_{\pi(n)}(x_0)\right|$$
and could then consider functions $\am^*$ of the type, say,
$$ \am^*(x_0, y) = \sum_{i=1}^{\left\lfloor n/3 \right\rfloor} \varepsilon_i \phi_{\pi(i)}(y) + \sum_{i=\left\lfloor n/3 \right\rfloor}^{n} \sgn( \phi_{\pi(i)}(x_0)) \phi_{\pi(i)}(y),$$
where $\varepsilon_i \in \left\{-1,1\right\}^n$ are now completely arbitrary. 
This will lead to an exponentially large number of functions ($\sim 2^{n/3}$) all of which could be used in the argument above.  In particular, $\phi_{n+1}$ has to have a structured inner product, a global correlation, with respect to \textit{all of them} since any single counterexample would be sufficient to obtain a bound on the eigenfunction $\phi_{n+1}$.

\subsection{An independence heuristic.}
This section describes a type of heuristic that is naturally suggested by these considerations and can be considered as being somewhat dual to Berry's random wave heuristic insofar as it is completely global while the random wave model is completely local. Suppose we are given $m$ distinct points $x_1, \dots, x_m \in M$ on the manifold. We think of these points as fixed as the number of eigenfunctions $n$ tends to infinity. Then, for any fixed $n \in \mathbb{N}$, we define $m$ distinct functions $f_1, \dots, f_m$ where $f_i: \left\{-1,1\right\}^n \rightarrow \mathbb{R}$ is given by, abbreviating $\varepsilon = (\varepsilon_1, \dots, \varepsilon_n) \in \left\{-1,1\right\}^n$,
$$ f_i(\varepsilon) = \frac{1}{\sqrt{n}}\sum_{k=1}^{n} \varepsilon_k \cdot \phi_k(x_i).$$
These $m$ functions $f_1, \dots, f_m$ can be interpreted as random variables defined on the space $\left\{-1,1\right\}^n$ (equipped with
the uniform measure). Because of the local Weyl law, we would expect that for $A \subset \mathbb{R}$ fixed, as $n \rightarrow \infty$,
$$ \mathbb{P}\left( f_i^{-1}(A)\right) = \frac{1}{2^n} \# \left\{\varepsilon \in \left\{-1,1\right\}^n: f_i(\varepsilon) \in A \right\} \rightarrow \frac{1}{\sqrt{2\pi}}\int_A e^{-x^2/2} dx.$$

It seems natural to conjecture that, on generic manifolds, these $m$ random variables behave, asymptotically as $n \rightarrow \infty$, like independent random variables: we would like that, for all $A_1, A_2, \dots, A_m \subset \mathbb{R}$, as $n \rightarrow \mathbb{N}$,
$$ \lim_{n \rightarrow \infty} \mathbb{P}\left( \bigcap_{i=1}^m f_i^{-1}(A_i)\right) = \frac{1}{(2\pi)^{m/2}} \prod_{i=1}^m \int_{A_i} e^{-x^2/2} dx.$$
Such a conjecture goes against spooky action at a distance in a very strong way. One would therefore not expect such a statement to be true on any manifold where explicit computation of eigenfunctions is possible. However, it could be an interesting problem to see whether any result in that direction could be rigorously established for a random basis on $\mathbb{S}^d$ or $\mathbb{T}^d$.

  \begin{center}
 \begin{figure}[h!]
 \begin{tikzpicture}
 \node at (0,0) {\includegraphics[width=0.4\textwidth]{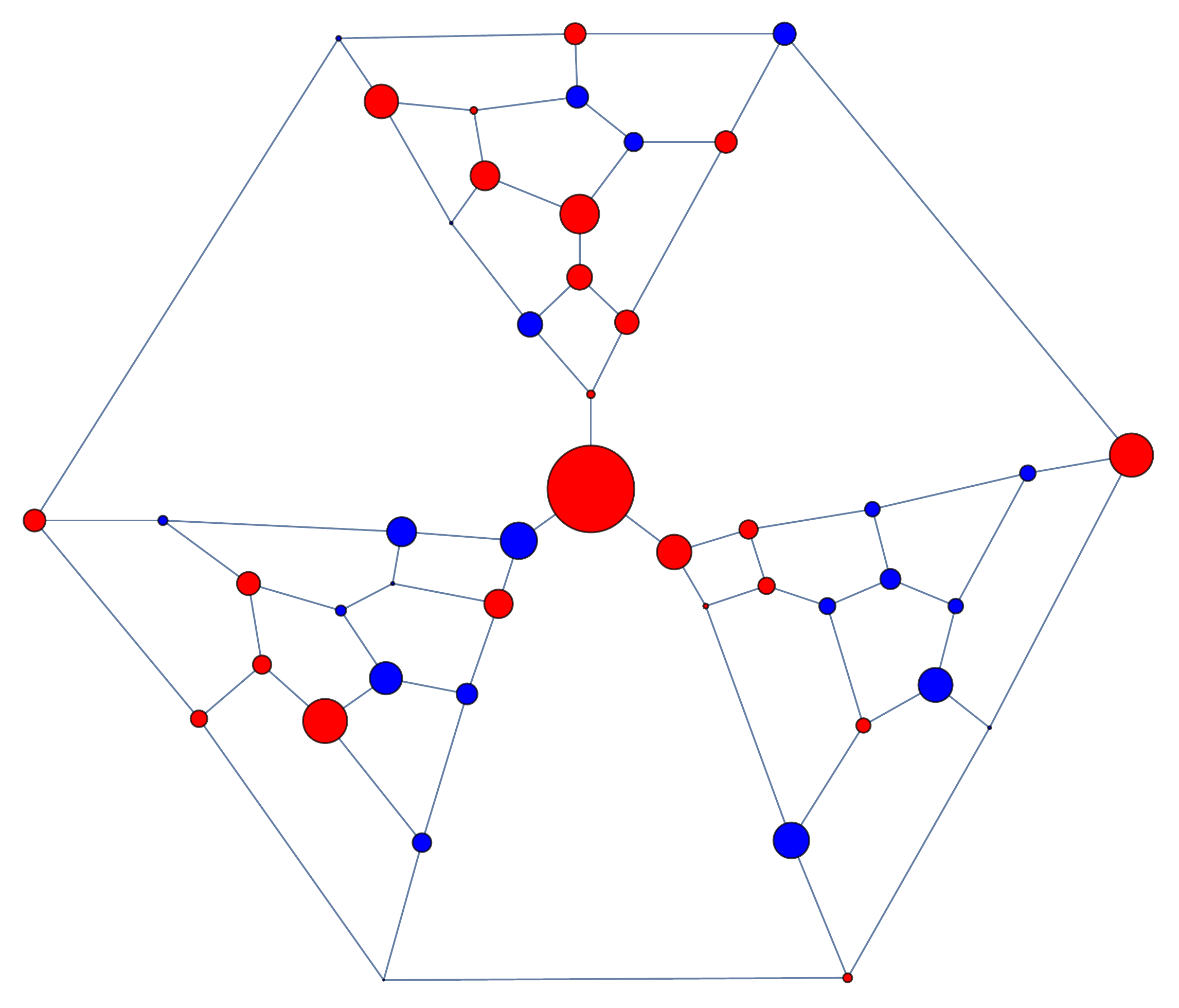}};
  \node at (5.45,0) {\includegraphics[width=0.4\textwidth]{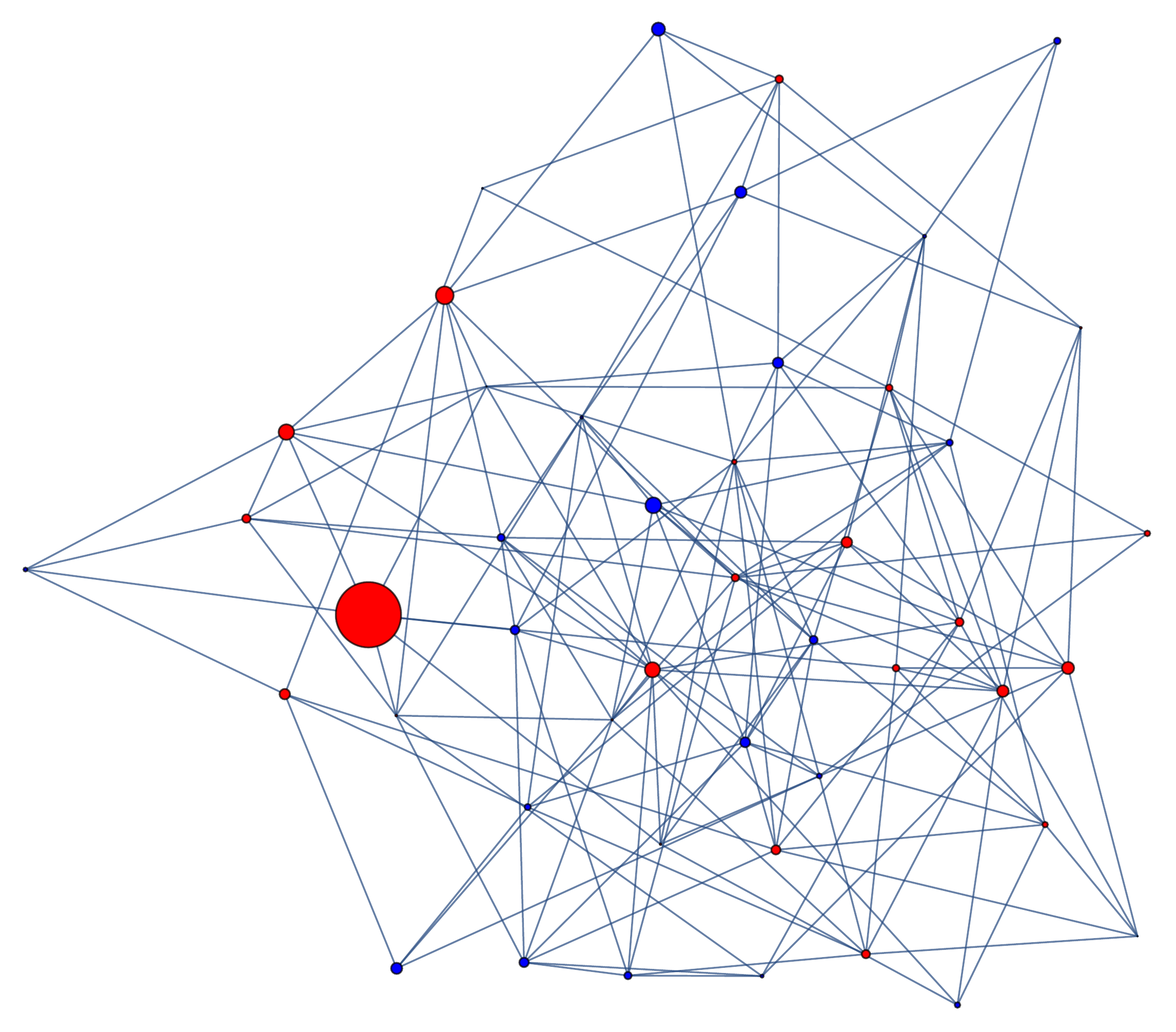}};
 \end{tikzpicture}
 \caption{$\am$ on the Tutte graph (left) and Erd\H{o}s-Renyi random graph (right). Color is determined by the sign, size is proportional to the absolute value. Structured graphs can lead to `non-asymptotic spooky action at a distance' while one would not expect this for unstructured graphs.}
 \label{fig:graph}
 \end{figure}
 \end{center}
 
\subsection{Graph Laplacians.} We note that Laplacian eigenfunctions have a natural analogue on graphs (see \ref{fig:graph}). Let $G=(V,E)$ be a graph on $|V| = n$ vertices. A discrete version of the Laplacian is the Laplacian matrix $ L = D - A$, defined by
$$ L_{ij} = \begin{cases} 
\mbox{deg}(v_i) \qquad &\mbox{if}~i = j  \\
 - 1_{E}(i,j) \qquad &\mbox{otherwise.} 
\end{cases}$$
The matrix $L \in \mathbb{R}^{n \times n}$ is real symmetric and therefore has real eigenvalues and eigenvectors which we denote by $(v_k)_{k=1}^{n}$.  We can thus define 
$$ \am^{(n)}(i, j) = \sum_{k=1}^{n} \sgn(v_k(i)) v_k(j).$$
Since our original argument connecting $\am$ and eigenfunction growth relied mainly orthogonality, it still can, in principle, be applied. In practice, it becomes a lot harder to understand even basic notions like the diagonal behavior $\am^{(n)}(i,i)$. Nonetheless, we believe that the underlying philosophy should still apply: eigenvector growth requires some underlying symmetries in the graph, possibly a degree of spooky action at a distance, which one would not expect at all to be the case for, say, random graphs (which are a particularly convenient model of a `generic' manifold). This is also aligned with the basic intuition that eigenfunction growth on graphs should be a rare phenomenon and possibly quite unstable since adding or removing a single edge may completely destroy certain symmetries.

\section{Proofs of the Propositions}

\subsection{Proof of Proposition 1}
\begin{proof}
We want to show that
\begin{align*}
\forall x \in M \qquad \int_M \am(x,y)^2 dy &\leq n \qquad \mbox{and} \qquad \int_M\int_M \am(x,y)^2 dy = n.
\end{align*}
These identities follow from basic orthogonality. Since
$$ \am(x,y) = \sgn(\phi_1(x)) \phi_1(y) +  \sgn(\phi_2(x)) \phi_2(y) + \dots +  \sgn(\phi_n(x)) \phi_n(y)$$
and these functions are orthonormal in $L^2$, we simply have
$$  \int_M \am(x,y)^2 dy = \sum_{k=1}^{n} \sgn(\phi_k(x))^2 \leq n.$$
The sum might be less than $n$ since in our definition $\sgn(0) = 0$. However, the set where the eigenfunctions
vanish is a set of measure 0 and since $\vol(M)=1$,
$$ \int_M  \sum_{k=1}^{n} \sgn(\phi_k(x))^2 dx= n.$$
\end{proof}

\subsection{Proof of Proposition 2}
\begin{proof}
Using $\vol(M) = 1$ and the $L^2-$normalization, 
\begin{align*}
\int_{M}  \am(x,x) dx =    \int_{M}  \sum_{k=1}^{n} |\phi_k(x)| dx = \sum_{k=1}^{n} \| \phi_k\|_{L^1} \leq \sum_{k=1}^{n} \| \phi_k\|_{L^2}  = n.
\end{align*}
The local Weyl law shows
$$ \sum_{k=1}^{n} |\phi_k(x)|  \leq \sqrt{n} \left( \sum_{k=1}^{n} \phi_k(x)^2 \right)^{1/2} \leq \sqrt{n} \sqrt{ n + \mathcal{O}(n^{\frac{d-1}{d}})} = n +  \mathcal{O}(n^{\frac{d-1}{d}}).$$
It remains to prove the lower bound
$$   \sum_{k=1}^{n} |\phi_k(x)| ~ \gtrsim_{(M,g)}~ n^{\frac{d+1}{2d}}.$$
Using the local Weyl law in combination with H\"ormander's bound, we arrive that
\begin{align*} n + \mathcal{O}(n^{\frac{d-1}{d}}) &=  \sum_{k=1}^{n} \phi_k(x)^2 \\
&\leq \left(\max_{1 \leq k \leq n} \|\phi_k\|_{L^{\infty}}\right) \sum_{k=1}^{n} |\phi_k(x)| \\
&\lesssim n^{\frac{1}{2} - \frac{1}{2d}} \cdot \sum_{k=1}^{n} |\phi_k(x)| 
\end{align*}
from which the desired claim follows.
\end{proof}

We also give an elementary proof of the weaker inequality
$$ \sum_{k=1}^{n} |\phi_k(x)|  \lesssim_d  n,$$
where the implicit constant depends only on the dimension. The advantage of the proof is that it is
purely elementary and avoids the local Weyl law (and some of the arguments will be useful in other settings).
\begin{proof}
We aim to derive the upper bound 
$$ \max_{x \in M} \am^{(n)}(x,x) \lesssim n.$$
Recalling
$$\int_{M} \am^{(n)}(x,x) dx = \int_M \sum_{k=1}^{n} |\phi_k(x)| dx = \sum_{k=1}^{n} \| \phi_k\|_{L^1} \leq \sum_{k=1}^{n} \| \phi_k\|_{L^2} \leq n,$$
 we see that, typically, $\am^{}(x,x)$ cannot be much larger than $n$.
For simple reasons of continuity, we expect that $\am^{}(x,y)$ should be comparable to $\am^{}(x,x)$ as long as $x$ is within one wavelength of $y$. Indeed, 
using $p_t(\cdot, \cdot)$ to denote the heat kernel associated to the eigenvalue problem,
\begin{align*}
 \int_{M} p_t(x,y) \am^{(n)}(x,y) dy &= \sum_{k=1}^{n} \sgn( \phi_k(x)) \int_M p_t(x,y) \phi_k(y) dx \\
 &= \sum_{k=1}^{n} \sgn( \phi_k(x)) e^{-\lambda_k t} \phi_k(x) \\
 &=  \sum_{k=1}^{n}e^{-\lambda_k t}  |\phi_k(x)| 
\geq e^{-\lambda_n t} \am(x,x).
 \end{align*}
 Suppose now that $\am^{(n)}(x,x)$ assumes its maximum in $x_0$. We will use
 $$ e^{-\lambda_n t} \am^{(n)}(x_0,x_0) \leq  \int_{M} p_t(x_0,x) \am^{(n)}(x_0,x) dx.$$
Introducing the set $A_t \subset M$ defined by
$$ A = \left\{x \in M: \am^{(n)}(x_0,x)\geq \frac{e^{-\lambda_n t}}{10} \am^{(n)}(x_0,x_0)  \right\}$$
allows us to argue that
 \begin{align*}
  e^{-\lambda_n t} \am^{}(x_0,x_0) &\leq  \int_{M} p_t(x_0,x) \am^{}(x_0,x) dx \\
  &= \int_{A_t} p_t(x_0,x) \am^{}(x_0,x) dx + \int_{M \setminus A_t} p_t(x_0,x) \am^{}(x_0,x) dx \\
  &\leq  \int_{A_t} p_t(x_0,x) \am^{}(x_0,x) dx + \frac{e^{-\lambda_n t}}{10} \am^{}(x_0,x_0)\int_{M \setminus A_t} p_t(x_0,x)  dx \\
  &\leq \int_{A_t} p_t(x_0,x) \am^{}(x_0,x) dx + \frac{e^{-\lambda_n t}}{10} \am^{}(x_0,x_0)
  \end{align*}
  and therefore
 $$  \int_{A_t} p_t(x_0,x) \am^{}(x_0,x) dx \geq \frac{9}{10} e^{-\lambda_n t} \am^{}(x_0,x_0).$$
 Combining this with the standard heat kernel bound
 $ p_t(\cdot, \cdot) \lesssim t^{-d/2}$
 implies 
 $$  \int_{A_t} \am^{}(x_0,x) dx \gtrsim e^{-\lambda_n t} t^{d/2} \am^{}(x_0,x_0).$$
 We now fix $t = 1/\lambda_n$. Then, using Weyl's law,
 $$ t^{d/2} = \lambda_n^{-d/2} \sim (n^{\frac{2}{d}})^{-\frac{d}{2}} \sim 1/n$$
 we see that for this choice of $t$ together with Cauchy-Schwarz
 $$  \frac{e^{-1}}{n}  \am^{}(x_0,x_0) \lesssim \int_{A_t} \am^{}(x_0,x) dx \leq |A_t|^{1/2} \left(  \int_{A_t} \am(x_0,x)^2 dx \right)^{1/2}  $$
We note that, due to $L^2-$normalization and orthogonality
$$ \int_{A_t} \am(x_0,x)^2 dx \leq \int_M \am(x_0,x)^2 dx \leq n$$
 which has two implications. First, it implies
  $$  \amalg(x_0,x_0) \lesssim  n^{3/2} |A_t|^{1/2}.$$
Secondly, recalling the definition of $A_t$, it shows that
$$n \geq  \int_{A_t} \amalg(x_0,x)^2 dx \geq \frac{\amalg(x_0,x_0)^2}{100} |A_t|$$
from which we infer
  $$  \amalg(x_0, x_0) \lesssim  n^{3/2} |A_t|^{1/2} \lesssim n^{3/2} \frac{\sqrt{n}}{\amalg(x_0,x_0)}$$
which is the desired result.
\end{proof}

\subsection{Proof of Proposition 3}
\begin{proof}
The local behavior of $\phi_i(x)$ around $x = \pi/2$ is simply given as follows
$$ \phi_i(x) \sim \begin{cases} 
 i (x- \pi/2) \qquad &\mbox{if}~i \equiv 0~(\mbox{mod}~4) \\
1 - \frac{i^2}{2} (x-\pi/2)^2 \qquad &\mbox{if}~i \equiv 1~(\mbox{mod}~4) \\
- i (x- \pi/2) \qquad &\mbox{if}~i \equiv 2~(\mbox{mod}~4) \\
-1 + \frac{i^2}{2} (x-\pi/2)^2 \qquad &\mbox{if}~i \equiv 3~(\mbox{mod}~4) 
\end{cases}$$
 where the expansions are accurate up to second order for $|x - \pi/2| \lesssim n^{-1}$. We will now consider the choice
 $$ x = \frac{\pi}{2} + \varepsilon,$$
 where $0 < \varepsilon \ll n^{-10}$. We first discuss a rough heuristic computation to get a sense for how things scale before making everything rigorous. We expect
 $$ \am^{(n)}(x,x) = \frac{n}{2} + o(n)$$
 and
\begin{align*}
 \am^{(n)}(x,y) &= \sum_{i=1 \atop i \equiv 0 ~(\tiny \mbox{mod}~4)}^{n} \sin{(iy)} + \sum_{i=1 \atop i \equiv 1 ~(\tiny \mbox{mod}~4)}^{n} \sin{(iy)} \\
 &- \sum_{i=1 \atop i \equiv 2 ~(\tiny \mbox{mod}~4)}^{n} \sin{(iy)} -  \sum_{i=1 \atop i \equiv 3 ~(\tiny \mbox{mod}~4)}^{n} \sin{(iy)}.
 \end{align*}
 We start with a local Taylor expansion up to second order
\begin{align*} \am^{(n)}(x,y) &\sim  \sum_{i=1 \atop i \equiv 0 ~(\tiny \mbox{mod}~2)}^{n} i (y - \pi/2) +  \sum_{i=1 \atop i \equiv 1 ~(\tiny \mbox{mod}~2)}^{n} 1- \frac{i^2}{2}(y-\pi/2)^2\\
&\sim \frac{n}{2} +  (y-\pi/2) \left(  \sum_{i=1 \atop i \equiv 0 ~(\tiny \mbox{mod}~2)}^{n} i \right) - \left(y - \frac{\pi}{2}\right)^2   \sum_{i=1 \atop i \equiv 1 ~(\tiny \mbox{mod}~2)}^{n} \frac{i^2}{2} \\
&\sim \frac{n}{2} +  \left(y - \frac{\pi}{2}\right) \frac{n^2}{4} -  \left(y - \frac{\pi}{2}\right)^2 \frac{n^3}{12}.
\end{align*}
which suggests that the $y - \pi/2 \sim c/n$. At that scale, higher-order terms in the Taylor expansion would contribute to the factor. We thus write, for $j \in \left\{0,2\right\}$ with $\varepsilon_0 = 1$ and $\varepsilon_2 = -1$ that
\begin{align*}
  \varepsilon_j \sum_{i=1 \atop i \equiv j ~(\tiny \mbox{mod}~4)}^{n} \sin{(iy)} &= \sum_{i=1 \atop i \equiv j ~(\tiny \mbox{mod}~4)}^{n} \sum_{\ell=1 \atop \ell~\mbox{\tiny odd}}^{\infty} (-1)^{\frac{\ell+3}{2}}\frac{i^{\ell}}{\ell!} \left( y - \frac{\pi}{2} \right)^{\ell} \\
  &=   \sum_{\ell=1 \atop \ell~\mbox{\tiny odd}}^{\infty}(-1)^{\frac{\ell+3}{2}}  \left( y - \frac{\pi}{2} \right)^{\ell}
\frac{1}{\ell!}  \sum_{i=1 \atop i \equiv j ~(\tiny \mbox{mod}~4)}^{n}i^{\ell}.
 \end{align*}
 Likewise, for $j \in \left\{1,3\right\}$ with $\varepsilon_1 = 1$ and $\varepsilon_3 = -1$, we have
 
 \begin{align*}
  \varepsilon_j \sum_{i=1 \atop i \equiv j ~(\tiny \mbox{mod}~4)}^{n} \sin{(iy)} &= \sum_{i=1 \atop i \equiv j ~(\tiny \mbox{mod}~4)}^{n}  1 - \sum_{\ell=1 \atop \ell~\mbox{\tiny even}}^{\infty} (-1)^{\frac{\ell+2}{2}} \frac{i^{\ell}}{\ell!} \left( y - \frac{\pi}{2} \right)^{\ell} \\
  &=  \mathcal{O}(1) + \frac{n}{4} - \sum_{\ell=1 \atop \ell~\mbox{\tiny even}}^{\infty} (-1)^{\frac{\ell+2}{2}} \left( y - \frac{\pi}{2} \right)^{\ell}
\frac{1}{\ell!}  \sum_{i=1 \atop i \equiv j ~(\tiny \mbox{mod}~4)}^{n}i^{\ell}.
 \end{align*}
 
Making the ansatz $x = \pi/2 + c/n$ as suggested by the first two terms of the Taylor expansion, we arrive at
\begin{align*}
  \varepsilon_j \sum_{i=1 \atop i \equiv j ~(\tiny \mbox{mod}~4)}^{n} \sin{(iy)} &=  \sum_{\ell=1 \atop \ell~\mbox{\tiny odd}}^{\infty} (-1)^{\frac{\ell+3}{2}} \left( y - \frac{\pi}{2} \right)^{\ell}
\frac{1}{\ell!}  \sum_{i=1 \atop i \equiv j ~(\tiny \mbox{mod}~4)}^{n}i^{\ell}\\
&= \sum_{\ell=1 \atop \ell~\mbox{\tiny odd}}^{\infty}(-1)^{\frac{\ell+3}{2}} \frac{c^{\ell}}{n^{\ell}}
\frac{1}{\ell!} \left( \frac{n^{\ell + 1}}{4(\ell + 1)} + \mathcal{O}(n^{\ell}) \right) \\
&= \mathcal{O}(1) + \frac{n}{4} \sum_{\ell=1 \atop \ell~\mbox{\tiny odd}}^{\infty} (-1)^{\frac{\ell+3}{2}}\frac{c^{\ell}}{(\ell+1)!}\\
&=  \mathcal{O}(1) + \frac{n}{4 c}  \sum_{\ell=2 \atop \ell~\mbox{\tiny even}}^{\infty} (-1)^{\frac{\ell + 2}{2}} \frac{c^{\ell}}{\ell!} \\
&=   \mathcal{O}(1) + \frac{n}{4 c} (1-\cos{c}).
 \end{align*}
 Likewise, for the second sum, we arrive at
 \begin{align*}
  \varepsilon_j \sum_{i=1 \atop i \equiv j ~(\tiny \mbox{mod}~4)}^{n} \sin{(iy)} &=  \mathcal{O}(1) + \frac{n}{4} - \sum_{\ell=1 \atop \ell~\mbox{\tiny even}}^{\infty} (-1)^{\frac{\ell+2}{2}} \frac{c^{\ell}}{n^{\ell}}
\frac{1}{\ell!}  \sum_{i=1 \atop i \equiv j ~(\tiny \mbox{mod}~4)}^{n}i^{\ell} \\
&=  \mathcal{O}(1) + \frac{n}{4} - \sum_{\ell=1 \atop \ell~\mbox{\tiny even}}^{\infty} (-1)^{\frac{\ell+2}{2}} \frac{c^{\ell}}{n^{\ell}}
\frac{1}{\ell!}  \left( \frac{1}{4} \frac{n^{\ell+1}}{(\ell+1)} + \mathcal{O}(n^{\ell}) \right) \\
&=   \mathcal{O}(1) + \frac{n}{4} - \frac{n}{4c} \sum_{\ell=1 \atop \ell~\mbox{\tiny even}}^{\infty} (-1)^{\frac{\ell+2}{2}} \frac{c^{\ell+1}}{(\ell+1)!}\\
&=   \mathcal{O}(1) + \frac{n}{4} - \frac{n}{4c} (c-\sin{c}).
 \end{align*}
 Summing over all four cases, we get
 $$ \am^{(n)}(x,y) = \mathcal{O}(1) + \frac{n}{2} + \frac{n}{4c}\left(2 - 2\cos{c} -2c +2\sin{c}\right).$$
 This expression is maximized for $c \sim 1.1750\dots$ for which we have
 $$ \am^{(n)}(x,y) \sim 1.30821 \cdot \am^{(n)}(x,x).$$
\end{proof}

\section{Proof of Theorem 1}
This section describes the proof of Theorem 1, the existence of spooky action at a distance for the unit interval with Dirichlet or Neumann boundary conditions and spooky action for the circle $\mathbb{S}^1$. We also discuss some aspects of $\am$ on $\mathbb{S}^2$ that are relevant for eigenfunction concentration (see the remarks in \S 2.6).

\subsection{The circle $\mathbb{S}^1$} This is already a very interesting case. Consider $\mathbb{S}^1 \cong [0,2\pi]$ with endpoints identified. The eigenvalues all have multiplicity 2:  there is no unique canonical basis. Indeed, as we will see, the behavior of $\am$ depends strongly on the basis chosen (which is only natural since the behavior of eigenfunctions $\phi_k$ depends on the basis of eigenfunction and not merely on the manifold). We start by analyzing the canonical basis.\\

\textit{The canonical basis.} We are interested in the behavior of
$$ \am^{(2n)}(x,y) = \sum_{k=1}^{n} \sgn(\sin(kx)) \sin(ky) + \sum_{k=1}^{n} \sgn(\cos(kx)) \cos(ky) $$

\begin{proposition}
The canonical basis exhibits spooky action at a distance: we have
$$ \am^{(2n)}\left(\frac{2 \pi}{3}, 0\right) = - \frac{n}{3} + \mathcal{O}(1).$$
\end{proposition}
\begin{proof}
Since $\sin{(k\cdot 0)} = 0$, only the first sum involving the cosine remains. We note that
$$ \cos{\left(k \cdot \frac{2\pi}{3}\right)} = \begin{cases} 1 \qquad &\mbox{if}~k \equiv 0~(\mbox{mod}~3) \\
-\frac{1}{2} \qquad &\mbox{otherwise.} \end{cases}$$
Thus, since $\cos{(k \cdot 0)} = 1$, we end up with
$$ \am^{(2n)}\left(\frac{2 \pi}{3}, 0\right) = \sum_{k=1}^{n} \sgn\left(\cos{\left(k \cdot \frac{2\pi}{3}\right)}\right)   = - \frac{n}{3} + \mathcal{O}(1).$$
\end{proof}

\textit{A randomized basis.} Due to the symmetries of $\mathbb{S}^1$, we may replace the pair of eigenfunctions $\sin{(kx)}$ and $\cos{(kx)}$ (both corresponding to eigenvalue $k^2$) by
$$\sin{(k(x-x_k))} \quad \mbox{and} \quad \cos{(k(x-x_k))},$$
where $x_k \in \mathbb{S}^1$ is completely arbitrary. This allows for a multitude of bases to be considered. Let us consider now explicitly
\begin{align*}
 \am^{(2n)}(x,y) &= \sum_{k=1}^{n} \sgn(\sin(k(x-x_k))) \sin(k(y-x_k)) \\
 &+ \sum_{k=1}^{n} \sgn(\cos(k(x-x_k))) \cos(k(y-x_k)),
 \end{align*}
where $x_1, \dots, x_n$ are sampled uniformly at random from $[0,2\pi]$. This model is simple enough that
explicit computations can be carried out. 

\begin{proposition} For this type of random basis, we have
$$ \mathbb{E} ~\am^{(2n)}(x,x) = \frac{4}{\pi} n.$$
For all $x \neq y$, we have (uniformly in $n$)
$$  \mathbb{E} ~\left|\am^{(2n)}(x,y)\right| \lesssim \frac{1}{|x-y|}.$$
We also have, for two universal constants $c_1, c_2$ and any $x \in \mathbb{S}^1$, that
$$ \mathbb{P}\left( \max_{|y-x| \geq 1/\sqrt{n}} \left| \am^{(2n)}(x,y)\right| \leq  c_1(\log{n})^{c_2}\sqrt{n}  \right) \rightarrow 1.$$
\end{proposition}
\begin{proof}
The diagonal statement follows from 
$$ \int_0^{2\pi} |\sin{(x)}| dx = 4.$$
The off-diagonal part is based on the identities
\begin{align*}
 \int_0^{2\pi} \sgn(\sin(k(x-x_k))) \sin(k(y-x_k))  d x_k &= 4 \cos{(k \cdot (y-x))} \\
  \int_0^{2\pi} \sgn(\cos(k(x-x_k))) \cos(k(y-x_k))  d x_k &= 4 \cos{(k \cdot (y-x))}.
\end{align*}

  \begin{center}
 \begin{figure}[h!]
 \begin{tikzpicture}
 \node at (0,0) {\includegraphics[width=0.4\textwidth]{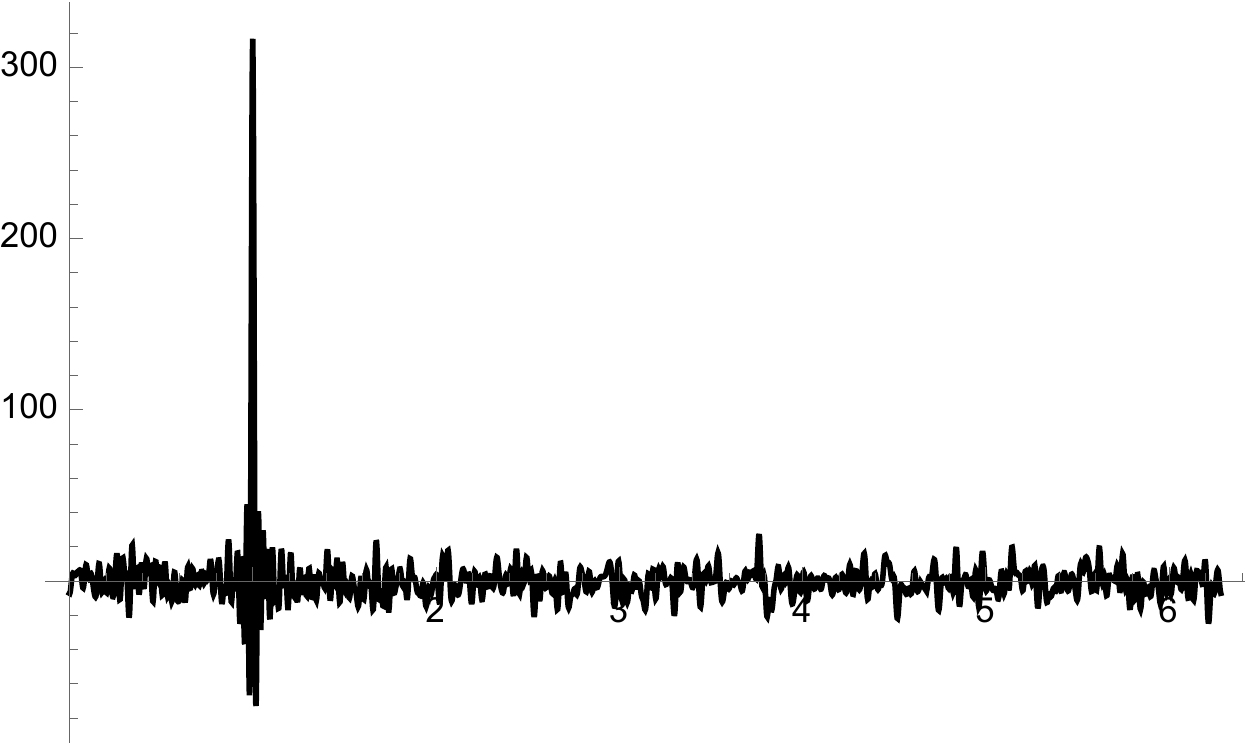}};
  \node at (6,0) {\includegraphics[width=0.4\textwidth]{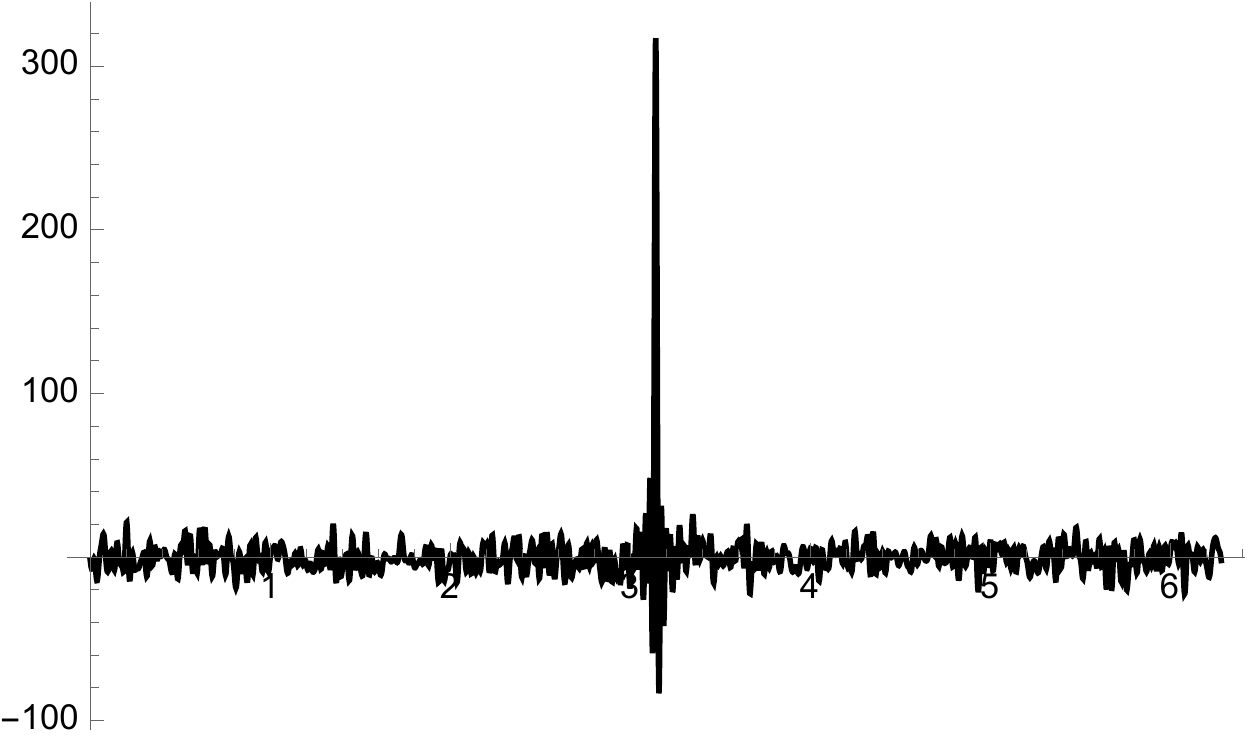}};
 \end{tikzpicture}
 \caption{Randomizing the basis destroys spooky actions: $\am^{(250)}(1,y)$ (left) and $\am^{(250)}(\pi, y)$ (right) for a fixed randomization of the Fourier basis.}
  \label{fig:destrospooky}
 \end{figure}
 \end{center}

From this we infer that
$$ \mathbb{E} \am^{(2n)}(x,y) = 8 \sum_{k=1}^{n} \cos{(k \cdot (x-y))}.$$
This sum is related to the classical Dirichlet kernel 
$$ \mathbb{E} \am^{(2n)}(x,y) = 8 \left( \frac{\sin{((n+1/2)(y-x))}}{\sin{(x/2)}} - 1\right).$$
We note that this explicit formula immediately implies that
$$  \mathbb{E} ~\left|\am^{(2n)}(x,y)\right| \lesssim \frac{1}{|x-y|}$$
as desired. For the concentration bound, we deal with both sums in isolation. Note that, for $x$ and $y \neq x$ fixed and $|y-x| \geq 1/\sqrt{n}$, we have
$$ \mathbb{E} \left|\sum_{k=1}^{n} \sgn(\sin(k(x-x_k))) \sin(k(y-x_k)) \right| \lesssim \sqrt{n}$$
from the above considerations. Moreover, since each of the summands is an independent random variable with variance bounded by $\sim 1$, we get that they are all tightly concentrated around the mean. The function
$$ f(y) = \sum_{k=1}^{n} \sgn(\sin(k(x-x_k))) \sin(k(y-x_k))$$
satsfies
$$ \|  f' \|_{L^{\infty}} \leq \sum_{k=1}^{n} k \leq n^2.$$
This means that if we want to enforce uniform smallness like $\|f\|_{L^{\infty}} \lesssim \sqrt{n \log{n}}$, it suffices to control whether this is satisfied at $\sim n^{3/2}$ equispaced points. The maximum of $\sim n^{3/2}$ Gaussian random variables is $\sim \log{n}$ and the result follows from the union bound.
\end{proof}

We have not tried to optimize the constants.
A similar approach might conceivably be possible on $\mathbb{T}^d$ or $\mathbb{S}^d$ (although a careful analysis might be much more difficult). This is much in line with the standard philosopy that on manifolds where eigenspaces have large multiplicity, a random basis should be representative of the behavior of `generic' eigenfunctions on `generic' manifolds.

 \subsection{The unit interval $[0,\pi]$.} The example in the previus section immediately implies spooky action at a distance for the unit interval
with Neumann boundary conditions. It remains to consider the unit interval with Dirichlet boundary conditions. The (unique) basis of
eigenfunctions is given by the sines and thus 
$$ \am^{(n)}(x,y) = \sum_{k=1}^{n} \sgn(\sin{(k  x)}) \sin{(k  y)}.$$

\begin{proposition} Suppose $0 < x < \pi/3$ and $x/\pi$ is irrational. Then
$$  \am^{(n)}(x,x) = \frac{2n}{\pi} + o(n)$$
and
$$ \am^{(n)}(x,3x) = \frac{2n}{3\pi} + o(n).$$
\end{proposition}
\begin{proof} We use an argument from \cite{stef}. We start by noting that if $x/\pi$ is irrational, then sequence $ (kx \mod 2\pi)_{k=1}^{\infty}$  is uniformly distributed on $[0, 2\pi]$. As a first implication, we obtain that
$$ \lim_{n \rightarrow \infty} \frac{1}{n} \sum_{k=1}^{n} \left|\sin{(kx)}\right| = \frac{1}{2\pi} \int_0^{2\pi} \left|\sin{(kx)}\right| dx = \frac{2}{\pi}.$$
This implies that the points $(kx, k (3x)) \in \mathbb{T}^2$ are uniformly distributed on the line shown in Figure \ref{fig:square}.

\begin{center}
\begin{figure}[h!]
\includegraphics[width=0.3\textwidth]{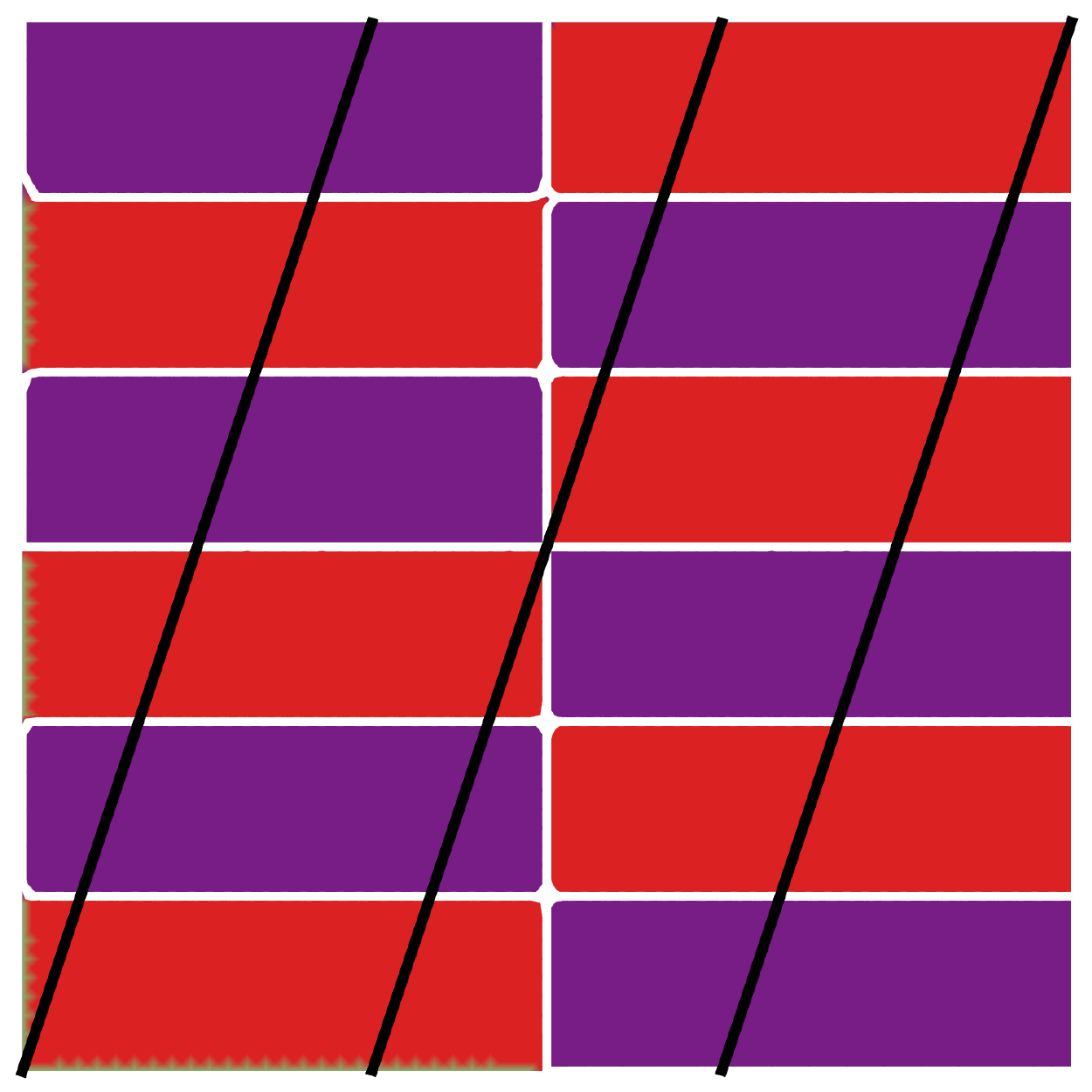}
\label{fig:square}
\caption{The function $\sgn( \sin(x) \sin(3y))$ on $[0,2\pi]^2$ (red $=1$, purple $=-1$) together with the line $y=3x$ on $[0,2\pi]^2$.}
\end{figure}
\end{center}

The line has a broken symmetry: its middle third segment spends two thirds of its time in red regions and only a third in purple regions from which we expect strong spooky action at a distance. Periodicity allows us to compute
$$ \lim_{n \rightarrow \infty} \frac{1}{n} \am^{(n)}(x,3x) =  \frac{1}{2\pi} \int_0^{2\pi} \sgn(\sin{(x)}) \sin{(3x)} dx = \frac{2}{3\pi}$$
from which the result follows. 
\end{proof}

\subsection{Zonal harmonics on $\mathbb{S}^2$}
The sphere, just as $\mathbb{T}$, has eigenvalues with large multiplicities. In particular, there are many possible basis and it is known \cite{van} that a `typical' basis is comprised of eigenfunctions whose $L^{\infty}-$norm grows only like a power of the logarithm of the frequency. However, in stark contrast,
the standard basis of eigenfunctions exhibits maximal eigenfunction growth and
\begin{equation} \label{eq:s2}
 \|\phi_k\|_{L^{\infty}} \lesssim \lambda_k^{\frac{1}{4}}
 \end{equation}
is attained for an infinite number of eigenfunctions. This case is, naturally, very well understood and we will rephrase it in terms of our framework. 
The complete set of eigenfunctions is given by
$$ \cos(m \phi) \cdot P_{l}^m(\cos{\theta}) \quad \mbox{and} \quad  \sin(m \phi) \cdot P_{l}^m(\cos{\theta})$$
where the eigenvalue is $\lambda = l(l+1)$, the value of $l \in \mathbb{N}_{\geq 0}$ and $m = 0,1,2,\dots, l$.
Here, $P_{l}^m$ are the associated Legendre polynomials. 
The extremal eigenfunctions with maximal growth are the zonal spherical harmonics: in the case of $\mathbb{S}^2$, the north
and south pole corresponds to $\cos{\theta} \in \left\{-1,1\right\}$. We easily see from the explicit formula of the eigenfunction that
in these points necessarily
$$ P_l^m = 0 \qquad \mbox{as soon as}~m \geq 1.$$
This reduces the analysis to the case $m=0$. Then
$$ P_l (\cos{\theta}) = P_l^0(\cos{\theta}) \quad \mbox{is just a classical Legendre Polynomial}$$
and we can set all the other signs as we want. We have
$$ P_l^0(\cos{0}) = 1.$$
The proper $L^2-$normalization is given by
$$ \phi_k(\theta) = \sqrt{k + \frac{1}{2}} \cdot P_k(\cos{\theta})$$
since, with that normalization,
$$ \int_0^{\pi} \phi_k(\theta)^2 \sin{(\theta)} d\theta = 1.$$
We also note that
$$ \left\| \phi_k \right\|_{L^{\infty}}= \phi_k(0) =   \sqrt{k + \frac{1}{2}}$$
while $-\Delta \phi_k = k(k+1) \phi_k$. The bound \eqref{eq:s2} is attained for all these eigenfunctions. Using $x$ to denote the
north pole, we see in particular that summing over the first $n$ eigenfunctions, the value of the eigenfunction will be usually 0
in the north pole and we only have a contribution for each individual eigenvalue. Therefore
$$\am^{(n)}(\mbox{north pole}, \mbox{north pole}) \sim \sum_{k=1}^{\sqrt{n}}  \sqrt{k + \frac{1}{2}} \sim n^{3/4}$$
which exactly saturates the universal lower bound for $\am$ on the diagonal \eqref{eq:univlower}. Another explicit computation, following from
$ P_l(\cos{0}) = (-1)^k,$
is
$$ \left| \am^{(n)}(\mbox{north pole}, \mbox{south pole}) \right| \sim  \left| \sum_{k=1}^{\sqrt{n}} (-1)^k \sqrt{k + \frac12}\right| \sim n^{1/4}.$$
It remains to understand the behavior in a generic point. Since $\am^{(n)}(\mbox{north pole}, y)$ is going to be a radial function, the value in
$y = (\theta, \phi)$ only depends on $\theta$ and
$$ \am^{(n)}(\mbox{north pole}, \theta) = \sum_{k=0}^n \sqrt{k + \frac12} \cdot P_k(\cos{\theta}).$$

\section{Proof of Theorem 2}
\begin{proof} We introduce, after possibly replacing without loss of generality $\phi_{n+1}$ by $-\phi_{n+1}$, the (not necessarily unique) point $z \in M$ so that
$$ \phi_{n+1}(z) = \| \phi_{n+1}\|_{L^{\infty}}.$$
Our goal will be to prove the desired inequality for $|\phi_{n+1}(z)|$ thus establishing it for $\| \phi_{n+1}\|_{L^{\infty}}$.
Since $\phi_{n+1}$ is an eigenfunction of the Laplacian, it's behavior under the heat equation is explicit and we can deduce that, for all $t > 0$,
$$ e^{-\lambda_{n+1} t} \phi_{n+1}(x) = \int_M p_t(x,y) \phi_{n+1}(y) dx.$$
The first part of our proof will be concerned with showing that the average of $\am(x,y)$ over $y$ at most a wavelength away from $x$ is locally comparable to $\am(x,x)$. For this, we recall the inequality (from the Proof of Theorem 1)
\begin{align*}
 \int_{M} p_t(x,y) \am^{(n)}(x,y) dy &= \sum_{k=1}^{n} \sgn( \phi_k(x)) \int_M p_t(x,y) \phi_k(y) dx \\
 &= \sum_{k=1}^{n} \sgn( \phi_k(x)) e^{-\lambda_k t} \phi_k(x) \\
 &=  \sum_{k=1}^{n}e^{-\lambda_k t}  |\phi_k(x)| 
\geq e^{-\lambda_n t} \am(x,x).
 \end{align*}
 The same argument also implies that
 $$  \int_{M} p_t(x,y) \am^{(n)}(x,y) dy  \leq \am(x,x).$$
We set $t =  \alpha/\lambda_n$ for a constant $0 < \alpha \ll 1$ to be determined. Then
 \begin{equation} \label{eq:ad}
  e^{-\alpha} \am(x,x) \leq \int_{M} p_t(x,y) \am^{(n)}(x,y) dy  \leq \am(x,x).
  \end{equation}
We will now evaluate the average of the product of $\am(z,y)$ and $\phi_{n+1}(y)$ weighted by the heat kernel centered at $z$. The
integral can be rewritten as
\begin{align*}
  \int_{M} p_t(z,y) \am(z,y) \phi_{n+1}(y) dy &=   \phi_{n+1}(z)   \int_{M} p_t(z,y) \am(z,y)  dy \\
 &+  \int_{M} p_t(z,y) \am(z,y) (\phi_{n+1}(y) - \phi_{n+1}(z))dy.
     \end{align*}
Using \eqref{eq:ad} and $\phi_{n+1}(z) = \| \phi_{n+1}\|_{L^{\infty}}$,
\begin{align*}
  \int_{M} p_t(z,y) \am(z,y) \phi_{n+1}(y) dy &\geq   \| \phi_{n+1}\|_{L^{\infty}}  e^{-\alpha} \am(z,z) \\
 &+  \int_{M} p_t(z,y) \am(z,y) (\phi_{n+1}(y) - \phi_{n+1}(z))dy.
     \end{align*}
As for analyzing the second integral: $p_t(\cdot, \cdot) \geq 0$ is always nonnegative and,
since $\phi_{n+1}$ assumes its maximum in $z$, we also have 
$$\phi_{n+1}(y) - \phi_{n+1}(z) \leq 0$$ 
and thus
$$ \forall~y \in M \qquad \qquad p_t(z,y) (\phi_{n+1}(y) - \phi_{n+1}(z)) \leq 0.$$
$\am(z,y)$ may assume either sign: since $y$ is pretty close to $z$ and $\am(z,z) > 0$, it stands to reason that the inequality
$$ \am(z,y) \leq \max_{w \in M} \am(z,w)$$
is not actually too lossy (see also \S 2.4). Combining all these estimates, we arrive at the lower bound

\begin{align*}
  \int_{M} p_t(z,y) \am(z,y) \phi_{n+1}(y) dy &\geq \| \phi_{n+1}\|_{L^{\infty}}  e^{-\alpha} \am(z,z) \\
 &+  \left(\max_{y \in M} \am(z,y) \right) \int_{M} p_t(z,y)(\phi_{n+1}(y) - \phi_{n+1}(z))dy.
     \end{align*}

Using once more than $\phi_{n+1}$ is an eigenfunction, we have
$$ \int_{M} p_t(z,y) \phi_{n+1}(y) dy = e^{-\lambda_{n+1} t} \phi_{n+1}(z) = e^{-\lambda_{n+1}t} \|\phi_{n+1}\|_{L^{\infty}}.$$
Finally, the heat kernel preserves constants and
$$  \int_{M} p_t(z,y)\phi_{n+1}(z)dy = \phi_{n+1}(z) = \|\phi_{n+1}\|_{L^{\infty}}.$$
Collecting these estimates, we arrive at

\begin{align*}
  \int_{M} p_t(z,y) \am(z,y) \phi_{n+1}(y) dy &\geq   \| \phi_{n+1}\|_{L^{\infty}}  e^{-\alpha} \am(z,z) \\
 &-  \left(\max_{y \in M} \am(z,y) \right) (1 - e^{-\alpha}) \|\phi_{n+1}\|_{L^{\infty}}.
     \end{align*}
Altogether, we obtain the lower bound
 $$ \int_{M} p_t(z,y) \am(z,y) \phi_{n+1}(y) dy \geq  X \cdot\| \phi_{n+1}\|_{L^{\infty}}$$
 where
 $$X =  e^{-\alpha} \left( \am(z,z) +  \max_{y \in M} \am(z,y) \right) -  \max_{y \in M} \am(z,y).$$
Naturally, for $\alpha$ small, this quantity is simply very close to $\am(z,z)$: this simply follows from continuity of
all the involved quantities. Indeed, for any $g \in C(M)$, 
$$ \lim_{t \rightarrow 0} \int_{M} p_t(x,y) g(y) dy = g(x).$$
The goal is to show that $\alpha$ need not be tremendously small for us to achieve the same effect up to constants.
We introduce the constant $c$ implicitly by
$$ \max_{y \in M} \am(z,y) = c \am(z,z).$$
Using the elementary inequality
$$ \forall~c \geq 1 \qquad e^{-\frac{1}{4c}} \geq  1 - \frac{1}{2c + 2},$$
we see that setting $\alpha = 1/(4c)$ implies that
$$ e^{-\alpha} \left(1 + c\right) - c \geq \frac{1}{2}.$$
 In particular, the choice $\alpha = 1/(4c)$ also implies
  $$ \int_{M} p_t(z,y) \am(z,y) \phi_{n+1}(y) dy \geq  \frac{1}{2} \am(z,z) \cdot\| \phi_{n+1}\|_{L^{\infty}}.$$
Standard heat kernel bounds imply
$$ \max_{w \in M} p_t(z,w) \lesssim t^{-\frac{d}{2}} \lesssim_c n$$
allowing us to bound
  $$ \int_{M} \frac{p_t(z,y)}{\max_{w \in M} p_t(z,w)} \am(z,y) \phi_{n+1}(y) dy \gtrsim_c  \frac{\am(z,z)}{2n}  \cdot\| \phi_{n+1}\|_{L^{\infty}}.$$
Orthogonality of eigenfunctions implies
$$ \int_{M} \am(z,y) \phi_{n+1}(y) dy = 0$$
and thus
$$\left|  \int_{M} \left(1 - \frac{p_t(z,y)}{\max_{w \in M} p_t(z,w)} \right) \am(z,y) \phi_{n+1}(y) dy \right| \gtrsim_c  \frac{\am(z,z)}{2n}  \cdot\| \phi_{n+1}\|_{L^{\infty}}.$$
This is exactly the desired statement.
\end{proof}

\end{document}